\documentclass[11pt,reqno,a4paper]{amsart}
\usepackage{setspace}
\usepackage[english]{babel}
\usepackage[utf8]{inputenc}
\usepackage{tikz}
\usetikzlibrary{cd}
\usepackage{amsfonts}  
\usepackage{amsmath} 
\usepackage{amsthm}
\numberwithin{equation}{section}
\usepackage{amssymb}
\usepackage{mathtools}
\usepackage{mathrsfs}
\usepackage{hyperref}
\usepackage{setspace}
\usepackage{csquotes}

\allowdisplaybreaks

\usepackage{units}
\usepackage{icomma}
\usepackage{graphicx}
\usepackage{tcolorbox}
\usepackage{listings}
\usepackage{relsize}
\usepackage{subcaption}
\usepackage{mdframed}
\usepackage{caption}
\usepackage{amsthm}
\usepackage{rotating}
\usepackage{tikz-cd}
\usepackage[margin=1in,footskip=.25in]{geometry}
\usepackage{parskip}
\usepackage{fancyhdr}

\patchcmd{\subsection}{-.5em}{.5em}{}{}
\patchcmd{\subsubsection}{-.5em}{.5em}{}{}

\makeatletter
\def\thm@space@setup{%
  \thm@preskip=\parskip \thm@postskip=0pt
}
\makeatother

\newcommand{\e}{\ensuremath{\mathrm{e}}} 

\newcommand{\Id}{\mathrm{Id}}
\newcommand{\pr}{\ensuremath{\mathrm{pr}}}

\newcommand{\Sch}{\mathrm{Sch}}

\renewcommand{\frak}{\mathfrak}

\newcommand{\N}{\mathbb{N}}
\newcommand{\Z}{\mathbb{Z}}
\newcommand{\Q}{\mathbb{Q}}
\newcommand{\R}{\mathbb{R}}
\newcommand{\C}{\mathbb{C}}

\newcommand{\PGL}{\mathrm{PGL}}

\newcommand{\Stab}{\mathrm{Stab}}
\newcommand{\Vol}{\mathrm{Vol}}

\newcommand{\Var}{\mathrm{Var}}
\newcommand{\Cov}{\mathrm{Cov}}

\newcommand{\rad}{\mathrm{rad}}
\newcommand{\sgn}{\mathrm{sgn}}

\newcommand{\qplusonetree}{\mathbb{T}_{q+1}}
\newcommand{\Gqplusone}{G_{q}}
\newcommand{\Kqplusone}{K_{q}}

\newcommand{\Emu}{\mathbb{E}_{\mu}}

\newcommand{\Varmu}{\mathrm{Var}_{\mu}}

\newcommand{\NVmu}{\mathrm{NV}_{\mu}}
\newcommand{\NV}{\mathrm{NV}}

\newcommand{\Poi}{\mathrm{Poi}}

\newcommand{\arccosh}{\mathrm{arccosh}}
\newcommand{\arcsinh}{\mathrm{arcsinh}}
\newcommand{\arctanh}{\mathrm{arctanh}}

\renewcommand{\Re}{\mathrm{Re}}
\renewcommand{\Im}{\mathrm{Im}}

\newcommand{\Aut}{\operatorname{Aut}}

\newcommand{\Borelinfty}{\mathscr{L}^{\infty}}
\newcommand{\Borelbndinfty}{\mathscr{L}_{\mathrm{fin}}}

\newcommand{\Radonplus}{\mathscr{M}_+}

\renewcommand{\hat}{\widehat}
\renewcommand{\tilde}{\widetilde}

\newcommand{\norm}[1]{\lVert#1\rVert}

\newcommand{\cH}{\mathcal{H}}
\newcommand{\cI}{\mathcal{I}}

\newcommand{\cO}{\mathcal{O}}
\newcommand{\cP}{\mathcal{P}}

\newcommand{\bH}{\mathbb{H}}

\newcommand{\bP}{\mathbb{P}}

\newcommand{\bS}{\mathbb{S}}
\newcommand{\bT}{\mathbb{T}}

\newcommand{\sH}{\mathscr{H}}

\newcommand{\sU}{\mathscr{U}}

\newcommand{\rS}{\mathrm{S}}

\definecolor{lichtgrijs}{gray}{1.00}
\mdfdefinestyle{mystyle}{ %
    backgroundcolor=lichtgrijs, %
    linewidth=0pt, %
    innertopmargin=5pt, %
    innerbottommargin=10pt,%
}

\newtheorem{theorem}{Theorem}[section]
\newtheorem{corollary}[theorem]{Corollary}
\newtheorem{proposition}[theorem]{Proposition}
\newtheorem{lemma}[theorem]{Lemma}

\theoremstyle{definition}
\newtheorem{definition}[theorem]{Definition}
\newtheorem{remark}[theorem]{Remark}

\newtheorem{example}{Example}[section]

\usepackage[bibstyle=numeric]{biblatex}

\renewbibmacro{in:}{}
\bibliography{Bibliography} 

\begin{document}

\newgeometry{top=3cm, bottom=3.5cm,left=3cm,right=3cm}

\title{Hyperuniformity in regular trees}
\author{Mattias Byl\'ehn}
\subjclass[2020]{Primary: 60G55. Secondary: 43A90, 05C48}
\keywords{Hyperuniformity, regular trees, spherical diffraction, asymptotic properties of spherical functions}

\pagestyle{plain}
\pagenumbering{arabic}

\begin{abstract}
We study notions of hyperuniformity for invariant locally square-integrable point processes in regular trees. We show that such point processes are never geometrically hyperuniform, and if the diffraction measure has support in the complementary series then the process is geometrically hyperfluctuating along all subsequences of radii. A definition of spectral hyperuniformity and stealth of a point process is given in terms of vanishing of the complementary series diffraction and sub-Poissonian decay of the principal series diffraction near the endpoints of the principal spectrum. Our main contribution is providing examples of stealthy invariant random lattice orbits in trees whose number variance grows strictly slower than the volume along some unbounded sequence of radii. These random lattice orbits are constructed from the fundamental groups of complete graphs and the Petersen graph.
\end{abstract}

\maketitle

\section{Introduction}
The notion of \emph{hyperuniformity} or \emph{superhomogeneity} introduced by Stillinger and Torquato in \cite{StillingerTorquato} is, in its mathematical formulation, the property of an invariant point process having sub-Poissonian number variance in the large-scale radial limit. An equivalent formulation is that the so called structure factor of the point process has sub-Poissonian decay in the small scale limit for the relevant frequency domain. This property has been studied extensively for invariant point processes in Euclidean spaces by many, and more recently in compact symmetric geometries as well as real hyperbolic spaces \cite{Björklund2024HyperuniformityOfRandomMeasuresOnEuclideanAndHyperbolicSpaces, Grabner2024HyperuniformPointsetsOnProjectiveSpaces, GrabnerHyperuniformityOnSpheresDeterministicAspects, GrabnerHyperuniformityOnSpheresProbabilisticAspects, StepanyukHyperuniformityOnFlatTori}.

In this paper we investigate notions of hyperuniformity and asymptotics of number variances of invariant point processes in regular trees. Similarly to the real hyperbolic setting in \cite{Björklund2024HyperuniformityOfRandomMeasuresOnEuclideanAndHyperbolicSpaces}, we prove for any invariant point process in a regular tree that there is an unbounded sequence of metric balls for which the number variance grows at least as fast as the volume of the balls in the radial limit. Moreover, we provide examples of random lattice orbits for which there is an unbounded sequence of balls along which the number variance grows strictly slower than the volume of the balls. We also provide random lattice orbits where the number variance grows as fast as the volume along any unbounded sequence of balls. The lattices involved are constructed from fundamental groups of finite regular graphs.

We also define the diffraction measure of an invariant point process in a regular tree, generalizing the notion of structure factor, and define spectral hyperuniformity of such a point process to mean that the diffraction measure has sub-Poissonian decay near the ends of the $L^2$-spectrum of the tree, and vanishes on the rest of the spectrum. An invariant point process is defined to be stealthy if the diffraction measure vanishes completely near the ends, and the lattices mentioned in the previous paragraph that we construct from finite regular graphs are stealthy in this sense. In particular, these finite graphs, for example complete and regular complete bipartite graphs, are \emph{Ramanujan graphs} in the sense of Lubotzky, Phillips and Sarnak \cite{LubotzkyPhillipsSarnak}. Every finite regular graph for which the random lattice orbit of the fundamental group is spectrally hyperuniform is a Ramanujan graph, and the converse holds if and only if the adjacency matrix of the Ramanujan graph does not admit eigenvalues at the endpoints of the $L^2$-spectrum.

\subsection{Motivation}

This paper is motivated by a question posed in \cite{Björklund2024HyperuniformityOfRandomMeasuresOnEuclideanAndHyperbolicSpaces} by Björklund and the author. There we show that, for any isometrically invariant locally square-integrable point process $\cP$ in real hyperbolic space $\bH^d$,
\begin{align*}
\limsup_{r \rightarrow +\infty} \frac{\Var(|\cP \cap B_r(o)|)}{\Vol_{\bH^d}(B_r(o))} > 0 \, . 
\end{align*}

\textbf{Question(s)}: \textit{Is there an (explicit) isometrically invariant locally square-integrable point process $\cP$ on $\bH^d$ such that}
\begin{align*}
\liminf_{r \rightarrow +\infty} \frac{\Var(|\cP \cap B_r(o)|)}{\Vol_{\bH^d}(B_r(o))} = 0 \, ? 
\end{align*}
\textit{Moreover, can $\cP$ be taken to be a random lattice orbit? Here, $B_r(o)$ denotes the metric ball of radius $r > 0$ centered at a fixed point $o \in \bH^d$.}

We expect that the answer to the first question is yes, and acknowledge the difficulty in answering the second question. In the same paper \cite[Section 5]{Björklund2024HyperuniformityOfRandomMeasuresOnEuclideanAndHyperbolicSpaces}, we prove that the invariant random lattice $\cP_{\Z^5}$ associated with the standard lattice $\Z^5$ in $\R^5$ satisfies the analogous result, namely
\begin{align*}
\liminf_{r \rightarrow +\infty} \frac{\Var(|\cP_{\Z^5} \cap B_r(o)|)}{\Vol_{4}(\partial B_r(o))} = 0 \, . 
\end{align*}
Here we answer the above question for regular trees. The following result is included in Theorem \ref{Theorem1.4}.
\begin{theorem}
\label{Theorem1.1}
The invariant random lattice orbit $\cP_{\mathfrak{K}_{4}}$ associated with the fundamental group of the tetrahedron graph $\mathfrak{K}_{4}$, acting on the $3$-regular 
 tree $\bT_3$ by deck automorphisms, satisfies 
\begin{align*}
\liminf_{r \rightarrow +\infty} \frac{\Var(|\cP_{\mathfrak{K}_{4}} \cap B_r(o)|)}{|B_r(o)|} = 0 \, . 
\end{align*}
Here, $B_r(o)$ denotes the metric ball of integer radius $r \geq 0$ centered at a fixed point $o \in \bT_3$.
\end{theorem}

\subsection{Point processes on regular trees}

Let $q \geq 2$ and consider the $(q + 1)$-regular tree $\qplusonetree$, which we will think of as a discrete metric space and let $\Radonplus(\qplusonetree)$ denote the space of positive locally finite measures on $\qplusonetree$. If $\cP$ is an $\Radonplus(\qplusonetree)$-valued random variable then the associated \emph{(weighted) point process} $\mu = \mu_{\cP}$ will to us simply be the law of $\cP$. A point process $\mu$ is \emph{invariant} if it is invariant under the action of the automorphism group $G_q = \Aut(\qplusonetree)$, and \emph{locally square-integrable} if 
\begin{align*}
\int_{\Radonplus(\qplusonetree)} p(\{x\})^2 d\mu(p) < +\infty
\end{align*}
for every $x \in \qplusonetree$. To study invariant locally square-integrable point processes $\mu$ one introduces for each finitely supported function $f : \qplusonetree \rightarrow \C$ a \emph{linear statistic} $\bS f : \Radonplus(\qplusonetree) \rightarrow \C$ by
\begin{align*}
\bS f (p) = \int_{\qplusonetree} f(x) \, dp(x) = \sum_{x \in \qplusonetree} p(\{x\}) f(x) \, . 
\end{align*}
Many interesting point processes are uniquely determined by the associated moments of these linear statistics, for example random lattice orbits, Poisson point processes and determinantal point processes. In the notation of the previous subsection, 
$$\Var(|\cP \cap B_r(o)|) = \Var_{\mu_{\cP}}(\bS\chi_{B_r(o)}) \, , $$
where $\chi_{B_r(o)}$ denotes the indicator function of the ball $B_r(o)$.

\subsection{The diffraction measure of a point process}

In order to do (spherical) harmonic analysis on regular trees, we identify $\qplusonetree$ with the homogeneous space $G_q/K_q$ where $G_q = \Aut(\qplusonetree)$ is the group of automorphisms of $\qplusonetree$ and $K_q$ is the stabilizer subgroup of a fixed root $o \in \qplusonetree$. On $G_q$ there is the \emph{spherical transform}, available for functions in the \emph{Hecke algebra} $\sH(G_q, K_q)$ of bi-$K_q$-invariant compactly supported functions $\varphi : G_q \rightarrow \C$, given by
\begin{align*}
\hat{\varphi}(\lambda) = \int_{G_q} \varphi(g) \omega_{\lambda}(g) dm_{G_q}(g) \, , \quad \lambda \in \C \, .
\end{align*}
Here $\omega_{\lambda}$ are the \emph{$K_q$-spherical functions} for $G_q$, 
\begin{align*}
\omega_{\lambda}(g) = c_q(\lambda) q^{-(1/2 - i\lambda)d(g.o, o)} + c_q(-\lambda) q^{-(1/2 + i\lambda)d(g.o, o)} \, ,
\end{align*}
and $c_q(\lambda) = \tfrac{1}{q^{1/2} + q^{-1/2}} \tfrac{q^{1/2 + i\lambda} - q^{-1/2 - i\lambda}}{q^{i\lambda} - q^{-i\lambda}}$ is the Harish-Chandra $c$-function for $G_q$.  

The \emph{diffraction measure} of an invariant locally square-integrable point process $\mu$ on $\qplusonetree$ is the unique positive finite Borel measure $\sigma_{\mu}$ supported on 
\begin{align*}
\Lambda_q = i[0, \tfrac{1}{2}] \cup (0, \tfrac{\pi}{\log(q)}) \cup (\tfrac{\pi}{\log(q)} + i[0, \tfrac{1}{2}]) \subset \C
\end{align*}
such that for all $\varphi_1, \varphi_2 \in \sH(G_q, K_q)$,
\begin{align*}
\langle \bS\varphi_1, \bS\varphi_2 \rangle_{L^2(\mu)} = \int_{\Lambda_q} \hat{\varphi}_1(\lambda) \overline{\hat{\varphi}_2(\lambda)} d\sigma_{\mu}(\lambda) \, . 
\end{align*}
The existence and uniqueness of this measure holds in the generality of random measures in commutative spaces $X = G/K$ associated to a locally compact second countable Gelfand pair $(G, K)$ with $K < G$ compact. These spaces include regular trees, Euclidean spaces, (higher rank) symmetric spaces as well as Bruhat-Tits buildings and products of such. We prove this in upcoming work with M. Björklund.
\begin{remark}[Connection to spectral graph theory]
There are some connections to representations and the spectrum of graphs to point out here.
\begin{enumerate}
    \item The set $\Lambda_q$ parameterizes all of the \emph{positive-definite} spherical functions $\omega_{\lambda}$, which in turn give rise to all equivalence classes of $K_q$-spherical irreducible unitary representations $\pi_{\lambda}$ of $G_q$. The representations coming from $\Im(\lambda) > 0$ are the \emph{complementary series} representations and $\Im(\lambda) = 0$ yields the \emph{principal series} representations. Special cases are the trivial representation for $\lambda = i/2$, the sign representation $\sgn(d(g.o, o))$ for $\lambda = \pi/\log(q) + i/2$, and for $\lambda = 0, \pi/\log(q)$ the corresponding spherical functions are analogues of the \emph{Harish-Chandra $\Xi$-function} in the setting of semisimple Lie groups.
    \item The set $\Lambda_q$ is also in bijection with the spectrum $\sigma(\Delta_q) = [-(q + 1), q + 1]$ of the adjacency operator $\Delta_q$ on $\qplusonetree$. Explicitly, the canonical bijection is
    $$ \Lambda_q \ni \lambda \longmapsto 2\sqrt{q} \cos(\log(q)\lambda) \in \sigma(\Delta_q) \, .  $$
    \item If $\mu = \Poi$ is the unit intensity invariant Poisson point process in $\qplusonetree$, then 
    \begin{align}
    \label{EqPoissonDiffraction}
        d\sigma_{\Poi}(\lambda) = \delta_{i/2} + \chi_{[0, \tfrac{\pi}{\log(q)}]}(\lambda) \tfrac{q^{1/2}}{q^{1/2} + q^{-1/2}} \tfrac{\log(q)}{2\pi}|c_q(\lambda)|^{-2} d\lambda
    \end{align}
    where the latter measure is the Plancherel measure for $G_q$. The Plancherel measure is in turn realized on the adjacency spectrum $\sigma(\Delta_q) = [-(q + 1), q + 1]$ as the celebrated \emph{Kesten-McKay measure}
    \begin{align*}
        \chi_{[-2\sqrt{q}, 2\sqrt{q}]}(\alpha) \frac{q + 1}{2\pi} \frac{\sqrt{4q - \alpha^2}}{(q + 1)^2 - \alpha^2} d\alpha \, . 
    \end{align*}
\end{enumerate}
\end{remark}

\subsection{Number variances}
Given an invariant locally square-integrable point process $\mu$ on $\qplusonetree$, its \emph{number variance} is the function $\NVmu : \Z_{\geq 0} \rightarrow \R_{\geq 0}$ given by the variance of the linear statistic of the indicator function on a centered metric ball,
\begin{align*}
\NVmu(r) = \Varmu(\bS\chi_{B_r(o)}) = q^r \int_{\Lambda_q \backslash \{\tfrac{i}{2}\}} \frac{(\sin_q(\lambda(r + 1)) + q^{-1/2}\sin_q(\lambda r))^2}{\sin^2_q(\lambda)} d\sigma_{\mu}(\lambda) \, .
\end{align*}
Due to the presence of $\pi/\log(q) + i/2 \in \Lambda_q$, corresponding to the spherical sign function and to the eigenvalue $-(q + 1)$ of the adjacency operator $\Delta_q$, it will be more convenient to consider what we will call the \emph{sub-oscillatory number variance} $\NVmu^* : \Z_{\geq 0} \rightarrow \R_{\geq 0}$ given by
\begin{align*}
\NVmu^*(r) = q^r \int_{\Lambda_q^*} \frac{(\sin_q(\lambda(r + 1)) + q^{-1/2}\sin_q(\lambda r))^2}{\sin^2_q(\lambda)} d\sigma_{\mu}(\lambda) \, , \quad \Lambda_q^* = \Lambda_q \backslash \{\tfrac{i}{2}, \tfrac{\pi}{\log(q)} + \tfrac{i}{2}\} \, . 
\end{align*}
In the case of random lattice orbits associated to fundamental groups of finite regular graphs, this will correspond to the removal of the eigenvalue $-(q + 1)$ of the adjacency operator when the graph is bipartite. 

\subsection{Asymptotics of number variances}

A first choice for defining geometric hyperuniformity of an invariant locally square-integrable point process $\mu$ on $\qplusonetree$ is to require that
\begin{align*}
\limsup_{r \rightarrow +\infty} \frac{\NV^*_{\mu}(r)}{|B_r(o)|} = 0  \, , 
\end{align*}
and to say that $\mu$ is geometrically hyperfluctuating if 
\begin{align*}
\limsup_{r \rightarrow +\infty} \frac{\NV^*_{\mu}(r)}{|B_r(o)|} = +\infty \, .   
\end{align*}
Similarly to the results in \cite[Theorem 1.4]{Björklund2024HyperuniformityOfRandomMeasuresOnEuclideanAndHyperbolicSpaces} for point processes in real hyperbolic spaces, we show that no such point processes $\mu$ are geometrically hyperuniform in this sense, and that the existence of complementary series diffraction yields a geometrically hyperfluctuating point process. The following can be seen as the analogue of Beck's Theorem in \cite[Theorem 2A]{BeckChenIrregularitiesOfDistribution} for regular trees.
\begin{theorem}[Non-geometric hyperuniformity]
\label{Theorem1.2}
Let $\mu$ be an invariant locally square-integrable point process in $\qplusonetree$. Then
\begin{align}
\label{EqNonGeometricHyperuniformity}
\limsup_{r \rightarrow +\infty} \frac{\NVmu^*(r)}{|B_r(o)|} > 0 \, . 
\end{align}
Moreover,
\begin{enumerate}
    \item if $\sigma_{\mu}(\{0, \tfrac{\pi}{\log(q)}\}) > 0$, then
    \begin{align*}
    \liminf_{r \rightarrow +\infty} \frac{\NVmu^*(r)}{r^2|B_r(o)|} > 0 \, . 
    \end{align*}
    \item if there is a $\delta \in (0, 1)$ such that $\sigma_{\mu}(i[\tfrac{\delta}{2}, \tfrac{1}{2})) > 0$ or $\sigma_{\mu}(\tfrac{\pi}{\log(q)} + i[\tfrac{\delta}{2}, \tfrac{1}{2})) > 0$, then
    \begin{align*}
    \liminf_{r \rightarrow +\infty} \frac{\NVmu^*(r)}{|B_r(o)|^{1 + \delta}} > 0 \, . 
    \end{align*}
\end{enumerate}
\end{theorem}
It follows that if $\mu$ admits complementary series diffraction, then it is geometrically hyperfluctuating along \emph{all} subsequences of radii by the above Theorem. In light of this Theorem, it is natural to investigate when
\begin{align*}
\liminf_{r \rightarrow +\infty} \frac{\NVmu^*(r)}{|B_r(o)|} = 0 
\end{align*} 
or not. We provide a criterion for when atoms in the diffraction picture yield a positive lower limit. To state it, consider the set
\begin{align*}
\Lambda_q^{\mathrm{Rat}} = \begin{cases} (0, \tfrac{\pi}{\log(2)}) \cap \tfrac{\pi}{\log(2)} \Q \backslash \{\tfrac{3\pi}{4\log(2)}, \tfrac{5\pi}{12\log(2)}, \tfrac{11\pi}{12\log(2)}\} &\mbox{ if } q = 2\\
(0, \tfrac{\pi}{\log(3)}) \cap \tfrac{\pi}{\log(3)} \Q \backslash \{\tfrac{5\pi}{6\log(3)}\} &\mbox{ if } q = 3\\
(0, \tfrac{\pi}{\log(q)}) \cap \tfrac{\pi}{\log(q)}\Q &\mbox{ if } q \geq 4 \, . 
\end{cases}
\end{align*}
\begin{theorem}
\label{Theorem1.3}
Let $\mu$ be an invariant locally square-integrable point process on $\qplusonetree$ with diffraction measure $\sigma_{\mu}$, and suppose that there is a $\lambda \in \Lambda_q^{\mathrm{Rat}}$ such that $\sigma_{\mu}(\{ \lambda \}) > 0$. Then
\begin{align*}
\liminf_{r \rightarrow +\infty} \frac{\NVmu^*(r)}{|B_r(o)|} > 0 \, . 
\end{align*} 
\end{theorem}
The proof requires a detailed analysis of trigonometric numbers, and the main tool is the rational linear independence of pairs of trigonometric numbers established by Berger in \cite[Theorem 1.2]{BergerOnLinearIndependenceOfTrigonometricNumbers}, extending a classical Theorem of Niven \cite[Corollary 3.12]{NivenIrrationalNumbers}.
\subsection{Random lattice orbits of fundamental groups} 
Given a finite simple connected $(q + 1)$-regular graph $\frak{X}$, its universal covering space is realized as the $(q + 1)$-regular tree $\qplusonetree$, and the fundamental group $\Gamma_{\frak{X}} < G_q$ of $\frak{X}$ is a cocompact lattice acting freely on $\qplusonetree$ by graph automorphisms. From this we obtain an invariant random lattice orbit $\mu_{\frak{X}} := \mu_{\Gamma_{\frak{X}}}$ with sub-oscillatory number variance $\NV^*_{\frak{X}} := \NV^*_{\mu_{\frak{X}}}$ and diffraction measure $\sigma_{\frak{X}} := \sigma_{\mu_{\frak{X}}}$ given by
\begin{align*}
\sigma_{\frak{X}} = \frac{1}{|\frak{X}|^2} \delta_{i/2} + \sum_{\lambda \in \Lambda_{\frak{X}}} m_{\frak{X}}(\lambda) \delta_{\lambda} \, .
\end{align*}
Here $\Lambda_{\frak{X}}$ denotes the finite subset of $\lambda \in \Lambda_q \backslash \{i/2\}$ such that $2\sqrt{q}\cos(\log(q)\lambda)$ is an eigenvalue of the adjacency operator $\Delta_{\frak{X}}$ on $\frak{X}$ and $m_{\frak{X}}(\lambda) > 0$ are multiplicity constants. Setting $\Lambda_{\frak{X}}^* = \Lambda_{\frak{X}} \cap \Lambda_q^*$ we can write the sub-oscillatory number variance as 
\begin{align*}
\NV_{\frak{X}}^*(r) = q^r \sum_{\lambda \in \Lambda_{\frak{X}}^*} m_{\frak{X}}(\lambda) \frac{(\sin_q(\lambda(r + 1)) + q^{-1/2}\sin_q(\lambda r))^2}{\sin^2_q(\lambda)} \, . 
\end{align*}
By Theorem \ref{Theorem1.2}, the random lattice orbit $\mu_{\frak{X}}$ is hyperfluctuating when $\Lambda_{\frak{X}} \cap i[0, \tfrac{1}{2}) \neq \varnothing$ or $\Lambda_{\frak{X}} \cap (\tfrac{\pi}{\log(q)} + i[0, \tfrac{1}{2})) \neq \varnothing$, and by Theorem \ref{Theorem1.3} we have that $\Lambda_{\frak{X}} \cap \Lambda_q^{\mathrm{Rat}} \neq \varnothing$ implies
\begin{align*}
\liminf_{r \rightarrow +\infty} \frac{\NV_{\frak{X}}^*(r)}{|B_r(o)|} > 0 \, . 
\end{align*}
We provide the following examples. 
\begin{theorem}
\label{Theorem1.4}
The random lattice orbit associated to
\begin{enumerate}
    \item the fundamental group $\Gamma_{\frak{K}_{q + 2}}$ of the complete $(q + 1)$-regular graph $\frak{K}_{q + 2}$ satisfies
     \begin{align*}
     \liminf_{r \rightarrow +\infty} \frac{\NV_{\frak{K}_{q + 2}}^*(r)}{|B_r(o)|} = 0 \, . 
     \end{align*}
     \item the fundamental group $\Gamma_{\frak{B}_{q + 1}}$ of the $(q + 1)$-regular complete bipartite graph $\frak{B}_{q + 1}$ satisfies
     \begin{align*}
     \liminf_{r \rightarrow +\infty} \frac{\NV_{\frak{B}_{q + 1}}^*(r)}{|B_r(o)|} > 0 \, . 
     \end{align*}
     \item the fundamental group $\Gamma_{\frak{P}_{3}}$ of the $3$-regular Petersen graph $\frak{P}_3$ satisfies
     \begin{align*}
     \liminf_{r \rightarrow +\infty} \frac{\NV_{\frak{P}_{3}}^*(r)}{|B_r(o)|} = 0 \, . 
     \end{align*}
\end{enumerate}
\end{theorem}
\begin{remark}
The set of non-trivial atoms for the diffraction measure in each of the cases in Theorem \ref{Theorem1.4} are
\begin{align*}
\Lambda_{\frak{X}} = \begin{cases} 
\{ \tfrac{1}{\log(q)} \arccos(- \tfrac{1}{2\sqrt{q}}) \} &\mbox{ if } \frak{X} = \frak{K}_{q+1} \\ 
\{ \tfrac{\pi}{2\log(q)}, \tfrac{\pi}{\log(q)} + \tfrac{i}{2} \} &\mbox{ if } \frak{X} = \frak{B}_{q+1} \\ 
\{ \tfrac{1}{\log(2)} \arccos(\tfrac{1}{2\sqrt{2}}), \frac{3\pi}{4\log(2)} \} &\mbox{ if } \frak{X} = \frak{P}_{3} \, .
\end{cases}
\end{align*}
In the proof of Theorem \ref{Theorem1.4}, we use equidistribution of irrationals on the unit circle for (1), Theorem \ref{Theorem1.3} for (2) since $\tfrac{\pi}{2\log(q)} \in \Lambda_q^{\mathrm{Rat}}$, and a similar equidistribution argument for (3). Note that the Peteresen graph $\frak{P}_3$ in (3) admits the atom $\frac{3\pi}{4\log(2)}$, which is on the list of exceptional rational multiples of $\tfrac{\pi}{\log(2)}$ outside $\Lambda_2^{\mathrm{Rat}}$ in Theorem \ref{Theorem1.3}.
\end{remark}
\begin{remark}
For case (2) in the above Theorem, the fundamental group of the bipartite graph $\frak{B}_{q + 1}$ is maximally hyperfluctuating with respect ordinary number variance as $\tfrac{\pi}{\log(q)} + \tfrac{i}{2} \in \Lambda_q$ is an atom of the diffraction measure $\sigma_{\frak{B}_{q+1}}$, yielding
\begin{align*}
     \liminf_{r \rightarrow +\infty} \frac{\NV_{\frak{B}_{q + 1}}(r)}{|B_r(o)|^2} > 0 \, . 
\end{align*}
\end{remark}

\subsection{Spectral hyperuniformity and stealth}

We define spectral hyperuniformity and stealth with respect to the diffraction measure $\sigma_{\Poi}$ of the unit intensity invariant Poisson point process from Equation \ref{EqPoissonDiffraction}. One can show that, for small $\varepsilon > 0$ we have 
$$\sigma_{\Poi}((0, \varepsilon] ) \asymp \varepsilon^3 \quad  \mbox{ and } \quad \sigma_{\Poi}([ \tfrac{\pi}{\log(q)} - \varepsilon, \tfrac{\pi}{\log(q)})) \asymp \varepsilon^3 \, . $$
\begin{definition}[Spectrally hyperuniform and stealthy point processes]
Let $\mu$ be an invariant locally square-integrable point process on $\qplusonetree$. We say that
\begin{itemize}
    \item $\mu$ is \emph{spectrally hyperuniform} if $\sigma_{\mu}(i[0, \tfrac{1}{2})) = \sigma_{\mu}(\tfrac{\pi}{\log(q)} + i[0, \tfrac{1}{2})) = 0$ and
    \begin{align*}
        \limsup_{\varepsilon \rightarrow 0^+} \varepsilon^{-3} \sigma_{\mu}\Big( (0, \varepsilon] \cup [ \tfrac{\pi}{\log(q)} - \varepsilon, \tfrac{\pi}{\log(q)}) \Big) = 0
    \end{align*}
    \item $\mu$ is \emph{stealthy} if $\sigma_{\mu}(i[0, \tfrac{1}{2})) = \sigma_{\mu}(\tfrac{\pi}{\log(q)} + i[0, \tfrac{1}{2})) = 0$ and there is an $\varepsilon_o > 0$ such that
    \begin{align*}
        \sigma_{\mu}\Big( (0, \varepsilon_o] \cup [ \tfrac{\pi}{\log(q)} - \varepsilon_o, \tfrac{\pi}{\log(q)}) \Big) = 0 \, . 
    \end{align*}
\end{itemize}
\end{definition}
Stealthy point processes are spectrally hyperuniform. The random lattice orbits in Theorem \ref{Theorem1.4} are all examples of stealthy point processes in $\qplusonetree$. More generally, a finite simple connected $(q + 1)$-regular graph $\frak{X}$ is \emph{Ramanujan} if the eigenvalues of the adjacency operator $\Delta_{\frak{X}}$ on $\frak{X}$, with the exception of $\pm (q + 1)$, are confined to $[-2\sqrt{q}, 2\sqrt{q}]$, meaning that the diffraction measure $\sigma_{\frak{X}}$ vanishes on the complementary series,
$$ \sigma_{\frak{X}}(i(0, \tfrac{1}{2})) = \sigma_{\frak{X}}(\tfrac{\pi}{\log(q)} + i(0, \tfrac{1}{2})) = 0  \, . $$
Thus the random lattice orbit $\mu_{\frak{X}}$ of the fundamental group $\Gamma_{\frak{X}}$ is stealthy if and only if $\frak{X}$ is Ramanujan and $\pm2\sqrt{q}$ is not an eigenvalue of the adjacency operator $\Delta_{\frak{X}}$. Excluding the eigenvalue $-(q + 1)$ of $\Delta_q$, or equivalently the spherical parameter $\pi/\log(q) + i/2 \in \Lambda_q$, is of particular importance since there are infinitely many regular bipartite Ramanujan graphs in every degree of any number of vertices as proven by Marcus, Spielman and Srivastava in \cite[Theorem 5.3]{SrivastavaInterlacingFamiliesI} and \cite[Theorem 1.2]{SrivastavaInterlacingFamiliesIV}. In fact, the graphs constructed in \cite[Theorem 1.2]{SrivastavaInterlacingFamiliesIV} result in stealthy random lattice orbits according to our definition. We take the stance that the random lattice orbits of the fundamental groups associated to such bipartite graphs should be stealthy rather than maximally hyperfluctuating.

We end this Introduction by posing a question regarding random lattice orbits of fundamental groups.

\textbf{Question}: \textit{What class of Ramanujan graphs $\frak{X}$ satisfy
\begin{align*}
\liminf_{r \rightarrow +\infty} \frac{\NV_{\frak{X}}^*(r)}{|B_r(o)|} = 0 \, ?
\end{align*}
}

\subsection{Structure of the paper} In Section \ref{The automorphism group of a regular tree} we define Schreier coset graphs associated to tdlcsc groups and realize the tree $\qplusonetree$ as a Schreier graph. We moreover provide a Cartan decomposition and a Haar measure for the automorphism group $\Aut(\qplusonetree)$. In Section \ref{Spherical harmonic analysis on regular trees} we introduce the Hecke algebra, spherical functions, the spherical transform and compute the spherical transform of the indicator function of a ball. In Section \ref{Invariant point processes, diffraction and number variances} we define invariant locally square-integrable point processes, diffraction measures and number variances. Theorem \ref{Theorem1.2} is proved in Section \ref{Non-geometric hyperuniformity} and Theorem \ref{Theorem1.3} is proved in Section \ref{Atoms in the diffraction picture}. We realize the universal cover of a finite regular graph and define fundamental groups in Section \ref{Hyperuniformity of fundamental groups of finite regular graphs}. There we also compute the diffraction of the random lattice orbit associated with a fundamental group and prove Theorem \ref{Theorem1.4}. 

\subsection{Notation}

We will write $\tau = 2\pi/\log(q)$ for integers $q \geq 2$. We will also consider $\tau$-periodic trigonometric functions and hyperbolic analogues, which we denote by 
$$\cos_q(\lambda) = \cos(\log(q)\lambda) = \frac{q^{i\lambda} + q^{-i\lambda}}{2} $$ 
and similarly for $\sin_q, \tan_q, \cot_q, \cosh_q, \sinh_q, \tanh_q, \coth_q$ and so on. The inverse functions are
$$\arccos_q(\alpha) = \tfrac{1}{\log(q)} \arccos(\alpha)$$ 
and similarly for $\arcsin_q, \arctan_q, \mathrm{arccot}_q, \arccosh_q, \arcsinh_q, \arctanh_q, \mathrm{arccoth}_q$ and so on.

\subsection{Acknowledgement} This paper is part of the author’s doctoral thesis at the
University of Gothenburg under the supervision of Michael Björklund. The author is grateful to Björklund for posing the question answered in Theorem \ref{Theorem1.1} and for valuable discussions during the writing of this paper. The author would also like to thank Adam Keilthy and Mykola Pochekai for useful discussions regarding the Appendix.


\begin{spacing}{0.1}
\tableofcontents
\end{spacing}

\section{The automorphism group of a regular tree}
\label{The automorphism group of a regular tree}

In Subsection \ref{Automorphism groups of graphs} we survey digraphs and their automorphism groups. We construct such digraphs from totally disconnected groups in Subsection \ref{Countable simple Schreier graphs} as Schreier coset graphs, the main example being regular trees. We proceed by introducing the Cartan decomposition, relevant function spaces and invariant measures in Subsections \ref{Sections and Cartan decomposition of Aut(Td)}, \ref{Function spaces over regular trees} and \ref{Invariant measures}. 

\subsection{Automorphism groups of graphs}
\label{Automorphism groups of graphs}

A \emph{directed graph} or \emph{digraph} is a pair $\frak{X} = (V_{\frak{X}}, E_{\frak{X}})$ where $V_{\frak{X}}$ is a finite or countable set of \emph{vertices} and $E_{\frak{X}} \subset V_{\frak{X}} \times V_{\frak{X}}$ is a subset of (directed) \emph{edges}. Such a digraph $\frak{X}$ is 
\begin{itemize}
    \item \emph{undirected} if $(x, y) \in E_{\frak{X}}$ implies $(y, x) \in E_{\frak{X}}$,
    \item \emph{loop-free} if $(x, x) \notin E_{\frak{X}}$ for all $x \in V_{\frak{X}}$,
    \item \emph{simple} if it is undirected and loop-free,
    \item \emph{regular} if the digraph is simple and there is a $d \in \N$ such that
    for every $x \in V_{\frak{X}}$,
    $$ \deg(x) := |\{ y \in V_{\frak{X}} \, | \, (x, y) \in E_{\frak{X}}\}| = d \, . $$
    If such a $d$ exists, we say that $\frak{X}$ is \emph{$d$-regular}. 
    \item \emph{connected} if for every pair of vertices $x, y \in V_{\frak{X}}$ there are vertices $x = v_1, ..., v_m = y$ such that $(v_i, v_{i + 1}) \in E_{\frak{X}}$ for all $i = 1, ..., m - 1$.
\end{itemize}
We will restrict our attention to simple regular connected digraphs $\frak{X}$, unless told otherwise, and refer to them as \emph{simple regular connected graphs}. 

Given a simple connected graph $\frak{X}$ we may regard it as a discrete metric space $(V_{\frak{X}}, d_{\frak{X}})$, where $d_{\frak{X}}(x, y)$ is the minimal length of all paths between vertices $x$ and $y$. For such graphs we will abbreviate $v \in \frak{X}$ for $v \in V_{\frak{X}}$. The set of edges can then be recovered as $E_{\frak{X}} = d_{\frak{X}}^{-1}(\{1\})$. The metric ball centered at a vertex $x \in \frak{X}$ of radius $r \in \Z_{\geq 0}$ is
\begin{align*}
B_r(x) = \{ y \in V_{\frak{X}} \, | \, d_{\frak{X}}(x, y) \leq r \}  \, . 
\end{align*}
%
%
A \emph{digraph homomorphism} between two digraphs $\frak{X}_1, \frak{X}_2$ is a map $g : V_{\frak{X}_1} \rightarrow V_{\frak{X}_2}$ such that $(g(x), g(y)) \in E_{\frak{X}_2}$ whenever $(x, y) \in E_{\frak{X}_1}$, and we denote such a digraph homomorphism by $g : \frak{X}_1 \rightarrow \frak{X}_2$. A \emph{digraph isomorphism} is a bijective digraph homomorphism whose inverse is a digraph homomorphism, and a digraph isomorphism from a digraph to itself is called a \emph{digraph automorphism}. The group of digraph automorphisms of a digraph $\frak{X}$ will be denoted by $\Aut(\frak{X})$, and we note that if $\frak{X}$ is simple then this set coincides with the group of isometries of the metric space $(V_{\frak{X}}, d_{\frak{X}})$. We will denote the (left) action of $\Aut(\frak{X})$ on the vertex set $V_{\frak{X}}$ by $g.x := g(x)$. A digraph $\frak{X}$ is \emph{vertex-transitive} if $\Aut(\frak{X})$ acts transitively on $V_{\frak{X}}$. In the vertex-transitive case we topologize $\Aut(\frak{X})$ by declaring for every finite subset $F \subset V_{\frak{X}}$ the basic open subset
\begin{align*}
\cO(F) = \Big\{g \in \Aut(\frak{X}) \, \Big| \, g.x = x \quad \forall \, x \in F \Big\} \, .  
\end{align*}
Then $\cO(F)$ is a closed subgroup. It is moreover compact, for if $g_n \in \cO(F)$ is an unbounded sequence and $x \in F$ then there is a $y \in V_{\frak{X}}$ such that $(g_n.y)_n$ is unbounded, so that 
$$ d_{\frak{X}}(x, y) = d_{\frak{X}}(x, g_n.y) \rightarrow +\infty $$
as $n \rightarrow +\infty$, a contradiction. In particular, the stabilizer subgroup $\cO(\{x\})$ is a maximal compact open subgroup in $\Aut(\frak{X})$ for every vertex $x \in V_{\frak{X}}$. In total we have that $\Aut(\frak{X})$ is a totally disconnected locally compact second countable (tdlcsc) group.

\subsection{Countable simple Schreier graphs}
\label{Countable simple Schreier graphs}

Let $G$ be a tdlcsc group with identity element $e$ and suppose that $K < G$ is a compact open subgroup, so that $G/K$ is at most countable. If $S \subset K \backslash G / K$ is finite, then we define the \emph{Schreier digraph} $\frak{X} = \Sch(G, K, S)$ to be the graph with vertices $V_{\frak{X}} = G/K$ and edges
\begin{align*}
E_{\frak{X}} = \Big\{ (gK,hK) \in G/K \times G/K \, \Big| \, Kg^{-1}hK \in S \Big\} \, . 
\end{align*}
Note that such a graph $\frak{X}$ is connected and vertex-transitive under the action of $G$. We will make the assumption that $S = S^{-1}$, so the graph $\frak{X}$ is undirected, and also that $KeK \notin S$, meaning that $\frak{X}$ is simple. Enumerating the members of $S$ by $Ks_1K, Ks_2K, \dots, Ks_{N}K$ with $N = |S|$, then for each $i \in \{ 1, 2, \dots, N \}$ there is an $n_i \in \N$ and distinct elements $t_{i1}, t_{i2}, \dots, t_{in_i} \in G_q$ such that
\begin{align}
\label{EqKKDoubleCosetDecomposition}
Ks_iK = \bigsqcup_{j = 1}^{n_i} t_{ij}K \, . 
\end{align}
Also, the cosets $t_{ij}K$ are unique up to left translation by $K$, and $K$ acts transitively on $T_i = \{ t_{ij}K : j = 1, \dots, n_i\}$ by permutations.   
Note that the degree of a vertex $gK \in V_{\frak{X}}$ is by the decomposition in Equation \ref{EqKKDoubleCosetDecomposition}
\begin{align*}
\deg(gK) &= \sum_{i = 1}^{N} |\{ hK \in G/K \, | \, Kg^{-1}hK = Ks_iK \}| = \sum_{i = 1}^{N} n_i \, . 
\end{align*}
In total, the Schreier graphs of the described form are simple, connected and regular of degree $d = \sum_{i = 1}^{N} n_i$.

Conversely, if $\frak{X}$ is a simple, connected, vertex-transitive $d$-regular graph with at most countably many vertices and we fix a vertex $o \in V_{\frak{X}}$, then setting $G = \Aut(\frak{X})$, $K = \Stab_{G}(o)$ and 
\begin{align*}
S = \Big\{ KsK \in K \backslash G / K \, \Big| \, (o, s.o) \in E_{\frak{X}} \Big\}
\end{align*}
realizes $\frak{X}$ as the Schreier graph of $(G, K, S)$ by the above construction.

\begin{example}[Regular trees]
\label{ExampleRegularTrees}
Let $q \geq 2$ and consider the $(q + 1)$-regular tree $\qplusonetree$ with root vertex $o$. Fix a geodesic $(x_n \in \qplusonetree : n \in \Z)$, in other words a bi-infinite sequence of distinct adjacent vertices, such that $x_0 = o$. Set $G_q = \Aut(\qplusonetree), K_q = \Stab_{G_q}(o)$ and let $S_q = \{K_qs_oK_q\}$ where $s_o \in G_q$ is the unique automorphism satisfying $s_o.x_i = x_{i+1}$ for all $i \in \Z$. Then $G_q$ is tdlcsc, $K_q$ is open and compact in $G_q$ and 
\begin{align}
\label{EquationKKDoubleCosetDecompositionFortheRegularTree}
K_qs_oK_q = \bigsqcup_{i = 1}^{q + 1} t_iK_q
\end{align}
where $t_1, \dots t_{q + 1} \in G_q$ are elements such that $t_i.o$ is adjacent to $o$ for every $i = 1, \dots, q + 1$, thus parameterising every neighbouring vertex of $o$. In particular, since $s_o^{-1}.o$ is adjacent to $o$, then $K_qs_o^{-1}K_q = K_qs_oK_q$. We conclude that $\qplusonetree = \Sch(G_q, K_q, S_q)$. Moreover, note that $K_qs_o^nK_q$ considered as a subset of $V_{\qplusonetree} = G_q/K_q$ is precisely the metric sphere $\partial B_n(o)$ of vertices of distance $n \in \Z_{\geq 0}$ to the root $o$.
\end{example}
\begin{remark}[Regular trees as spaces of $\Z_p$-modules over $\Q_p$]
When $q = p$ for some prime $p$, there is a smaller group of isometries of a linear algebraic nature acting transitively on $\bT_{p+1}$. The set of equivalence classes of rank 2 $\Z_p$-modules in the $p$-adic vector space $\Q_p^2$ up to homothety form a $p$-regular tree by declaring edges to be pairs $([\Lambda_1], [\Lambda_2])$ satisfying $p\Lambda_1 \subset \Lambda_2 \subset \Lambda_1$. The projective general linear group $\PGL_2(\Q_p)$ acts transitively on this tree and defines a subgroup of $G_p$ with stabilizer $\PGL_2(\Z_p)$ in $K_p$. Thus for $p$ prime, we can identify the tree $\bT_{p + 1} = G_p/K_p$ with $\PGL_2(\Q_p)/\PGL_2(\Z_p)$. In this sense one can view regular trees as combinatorial/non-Archimedean analogues of the hyperbolic plane.
\end{remark}
We will in the following sections restrict our attention to the $(q + 1)$-regular tree $\qplusonetree$ and return to other graphs in Section \ref{Hyperuniformity of fundamental groups of finite regular graphs}.

\subsection{Sections and Cartan decomposition of $\Aut(\bT_d)$}
\label{Sections and Cartan decomposition of Aut(Td)}

Let $(\qplusonetree, o)$ denote the rooted $(q + 1)$-regular tree, let $d : \qplusonetree \times \qplusonetree \rightarrow \Z_{\geq 0}$ be the canonical graph metric and let $G_q, K_q, S_q$ be as in Example \ref{ExampleRegularTrees}. By the decomposition in Equation \ref{EquationKKDoubleCosetDecompositionFortheRegularTree}, the stabilizer $K_q$ of the root $o$ acts transitively on the spheres $\partial B_r(o)$ for every $r \geq 0$, so for every $x \in V_{\qplusonetree}$ we can find a rotation $k_x \in K_q$ such that 
$$x = k_xs_o^{d(x, o)}.o \, . $$
The map $\varsigma : \qplusonetree \rightarrow G_q$ given by $\varsigma(x) = k_xs_o^{d(x, o)}$ is then a continuous section of the orbit map $G_q \rightarrow \qplusonetree$. With the existence of such a section, we arrive at the \emph{Cartan decomposition} 
\begin{align} 
\label{EquationCartanDecomposition}
G_q = K_q A_q^+ K_q \, ,
\end{align}
where $A_q^+ = \{s_o^n \in G_q : n \in \Z_{\geq 0}\}$. The Cartan subgroup $A_q = \{ s_o^n \in G_q : n \in \Z\}$ retrieves the fixed geodesic $(x_n \in \qplusonetree : n \in \Z)$ used to define $s_o \in G_q$ in Example \ref{ExampleRegularTrees} by $s_o^n.o = x_n$ for all $n \in \Z$. An important consequence of the Cartan decomposition is the following double coset property.
\begin{lemma}[Inversion of $K_q$-by-$K_q$ cosets]
\label{LemmaInversionofKbyKCosets}
For every $g \in G_q$, $K_qg^{-1}K_q = K_qgK_q$.
\end{lemma}
\begin{proof}
Set $n = d(g.o, o)$ and note that $d(g^{-1}.o, o) = d(o, g.o) = n$. Since $K_q$ acts transitively on the sphere $\partial B_n(o)$, there is a $k_1 = k_1(g) \in K_q$ such that $g^{-1}.o = k_1g.o$, and by definition of $K_q$ as the stabilizer of $o \in \qplusonetree$ there is a $k_2 = k_2(g) \in K_q$ such that $g^{-1} = k_1 g k_2$. 
\end{proof}
Before moving on, we will write $B_r$ for the metric ball $B_r(o)$ and also for the \emph{centered Cartan ball} of radius $r \geq 0$,
\begin{align*}
B_r = \big\{ g \in G_q \, \big| \, d(g.o, o) \leq r \big\} = \big\{ k_1 s_o^{n} k_2 \in G_q \, \big| \, 0 \leq n \leq r \, , \,\, k_1, k_2 \in K_q \big\} \, . 
\end{align*}
With this convention, $B_r.o = B_r(o)$ in $\qplusonetree$.

\subsection{Function spaces over regular trees}
\label{Function spaces over regular trees}

In the subsequent sections, we will make special use of the vector space of finitely supported complex-valued functions on $\qplusonetree$, 
\begin{align*}
\Borelbndinfty(\qplusonetree) = \Big\{ f : \qplusonetree \rightarrow \C \, \Big| \, f(x) = 0 \mbox{ for all but finitely many } x \in \qplusonetree \Big\}
\end{align*}
as well as the space of bounded functions
\begin{align*}
\Borelinfty(\qplusonetree) = \Big\{ f : \qplusonetree \rightarrow \C \, \Big| \, \sup_{x \in \qplusonetree} |f(x)| < +\infty \Big\} \, . 
\end{align*}
A function $f : \qplusonetree \rightarrow \C$ is \emph{radial} if $f(k.x) = f(x)$ for all $k \in K_q$ and all $x \in \qplusonetree$, and we denote by $\Borelbndinfty(\qplusonetree)^{\rad}, \Borelinfty(\qplusonetree)^{\rad}$ the corresponding subspaces of radial functions. The automorphism group $G_q$ acts on these function spaces from the left by $(g.f)(x) = f(g^{-1}.x)$, and given a function $f : \qplusonetree \rightarrow \C$ we can lift it to a right-$K_q$-invariant function $\varphi_f(g) = f(g.o)$ on $G_q$.

\subsection{Invariant measures}
\label{Invariant measures}

It is clear that the counting measure 
\begin{align*}
m_{\qplusonetree}(f) = \sum_{x \in \qplusonetree} f(x) \, , \quad f \in \Borelbndinfty(\qplusonetree) 
\end{align*}
is invariant under $\Gqplusone$. Denoting the Haar probability measure on the compact group $\Kqplusone$ by $m_{\Kqplusone}$, we may write
\begin{align*}
m_{\qplusonetree}(f) = \int_{\Kqplusone} \Big(\sum_{n = 0}^{\infty}  f(ks_o^n.o) |\partial B_n(o)| \Big) dm_{\Kqplusone}(k) \, , 
\end{align*}
where
\begin{align*}
|\partial B_n(o)| = \begin{cases} 1 &\mbox{ if } n = 0 \\ \frac{q + 1}{q} q^n &\mbox{ if } n \geq 1 \, .  \end{cases}
\end{align*}
By the Cartan decomposition in Equation \ref{EquationCartanDecomposition} we obtain a (left) Haar measure on $\Gqplusone$ by defining
\begin{align*}
m_{\Gqplusone}(\varphi) = \int_{\Kqplusone} \sum_{n = 0}^{\infty}  \int_{\Kqplusone} \varphi(k_1s_o^nk_2) dm_{\Kqplusone}(k_2) |\partial B_n(o)| dm_{\Kqplusone}(k_1) 
\end{align*}
for any locally constant function $\varphi : G_q \rightarrow \C$. In particular, one sees from this formula that $\Gqplusone$ is unimodular. Moreover, the measure of the centered Cartan ball in $G_q$ is 
\begin{align*}
m_{G_q}(B_r) = \sum_{n = 0}^r |\partial B_n(o)| = |B_r(o)| = \begin{cases} 1 &\mbox{ if } r = 0 \\ \tfrac{q + 1}{q - 1} q^r - \tfrac{2}{q - 1} &\mbox{ if } r \geq 1 \, ,  \end{cases}
\end{align*}
which grows as $\tfrac{q + 1}{q - 1} q^r$ with $r \rightarrow +\infty$.

\section{Spherical harmonic analysis on regular trees}
\label{Spherical harmonic analysis on regular trees}

The main goal of this Section is to introduce and compute the spherical transform of indicator functions on balls. In Subsection \ref{The Hecke algebra and spherical functions} we introduce spherical functions as complex-valued $*$-homomorphisms on the Hecke algebra and derive explicit formulas. The positive-definite spherical functions are identified in Subsection \ref{Spherical representations and positive-definite spherical functions} and the spherical transform along with the corresponding Plancherel formula are given in Subsection \ref{The spherical transform}. Lastly, we compute the spherical transform of indicator functions on balls in Subsection \ref{The spherical transform of the indicator function of a ball}, alongside some useful asymptotic formulas for later sections.

\subsection{The Hecke algebra and spherical functions}
\label{The Hecke algebra and spherical functions}

The \emph{Hecke algebra} $\sH(G_q, K_q)$ of the pair $(G_q, K_q)$ is the involutive complex algebra of compactly supported bi-$K_q$-invariant functions $\varphi : G_q \rightarrow \C$ with the operation of convolution
\begin{align*}
(\varphi_1 * \varphi_2)(g) = \int_{G_q} \varphi_1(gh^{-1})\varphi_2(h) dm_{G_q}(h)
\end{align*}
and the involution $\varphi^*(g) = \overline{\varphi(g^{-1})}$ for all $g \in G_q$. One shows by standard arguments that $\sH(G_q, K_q)$ is preserved under these operations. The identity element is the indicator function of $K_q$, $\chi_{K_q} \in \sH(G_q, K_q)$. The key feature of the Hecke algebra for us is that it enjoys commutativity. 
\begin{lemma}
The Hecke algebra $\sH(G_q, K_q)$ is commutative under convolution. 
\end{lemma}
\begin{proof}
By Lemma \ref{LemmaInversionofKbyKCosets}, every $\varphi \in \sH(G_q, K_q)$ satisfies $\varphi(g^{-1}) = \varphi(g)$ for all $g \in G_q$. Thus if $\varphi_1, \varphi_2 \in \sH(G_q, K_q)$ then
\begin{align*}
(\varphi_1 * \varphi_2)(g) &= \int_{G_q} \varphi_1(gh^{-1})\varphi_2(h) dm_{G_q}(h) = \int_{G_q} \varphi_1(gh^{-1})\varphi_2(h^{-1}) dm_{G_q}(h) \\
&= \int_{G_q} \varphi_1(h^{-1})\varphi_2(g^{-1}h^{-1}) dm_{G_q}(h) = \int_{G_q} \varphi_1(h)\varphi_2(g^{-1}h^{-1}) dm_{G_q}(h)\\
&= (\varphi_2 * \varphi_1)(g^{-1}) = (\varphi_2 * \varphi_1)(g) \, . 
\end{align*}
\end{proof}
We say that $(G_q, K_q)$ is a \emph{Gelfand pair}. Note that the bi-$K_q$-invariance of a function $\varphi \in \sH(G_q, K_q)$ ensures that $\varphi$ is locally constant and hence continuous on $G_q$ since $K_q$ is open. More precisely, by the Cartan decomposition every $\varphi \in \sH(G_q, K_q)$ can be written as 
\begin{align*}
\varphi = \sum_{n = 0}^{\infty} \varphi(s_o^n) \chi_{K_qs_o^nK_q} \, , 
\end{align*}
which is continuous.
\begin{lemma}
\label{LemmaHeckeAlgebraGenerator}
The function $\chi_o = \chi_{K_q s_o K_q}$ generates the algebra $\sH(G_q, K_q)$.
\end{lemma}
\begin{proof}
Since the indicator functions $\chi_{K_qs_o^nK_q}$ with $n \in \Z_{\geq 0}$ span the Hecke algebra $\sH(G_q, K_q)$ then it suffices to prove that they lie in the unital algebra generated by $\chi_{K_qs_oK_q}$. First, note that 
\begin{align*}
(\chi_{K_qs_oK_q} * \chi_{K_qs_o^nK_q})(g) &= \int_{K_qs_oK_q} \chi_{K_qs_o^nK_q}(gh^{-1}) dm_{G_q}(h) \\
&= (q + 1) \int_{K_q} \chi_{K_qs_o^nK_q}(gks_o^{-1}) dm_{K_q}(k) \, . 
\end{align*}
Note that $(g.o, gks_o^{-1}.o) \in E_{\qplusonetree}$ for all $k \in K_q$, so that
\begin{align*}
(\chi_{K_qs_oK_q} * \chi_{K_qs_o^nK_q})(g) &= (q + 1) \int_{K_q} \chi_{K_qs_o^nK_q}(gks_o^{-1}) dm_{K_q}(k) \\
&= \sum_{(g.o, y) \in E_{\qplusonetree}}  \chi_{K_qs_o^nK_q}(s_o^{d(y, o)}) \\
&= q \chi_{K_qs_o^{n + 1}K_q}(s_o^{d(g.o, o)}) + \chi_{K_qs_o^{n-1}K_q}(s_o^{d(g.o, o)}) \\
&= q \chi_{K_qs_o^{n + 1}K_q}(g) + \chi_{K_qs_o^{n-1}K_q}(g)\, . 
\end{align*}
In particular, we have that 
\begin{align*}
\chi_{K_qs_o^2K_q} = \frac{1}{q}\Big( \chi_{K_qs_oK_q}^{* 2} - \chi_{K_q} \Big) \, , 
\end{align*}
and proceeding by induction on $n$ we get that 
$$\chi_{K_qs_o^nK_q} \in \mathrm{span}_{\C}\big\{\chi_{K_q}, \chi_{K_qs_oK_q}, \chi_{K_qs_oK_q}^{* 2}, ..., \chi_{K_qs_oK_q}^{* n}\big\} \, . $$
\end{proof}

We define the $K_q$-spherical functions on $G_q$ as the functions arising from non-zero multiplicative linear functionals from the Hecke algebra $\sH(G_q, K_q)$ to $\C$. 

\begin{definition}[Spherical functions]
A function $\omega : G_q \rightarrow \C$ is \emph{$K_q$-spherical} if it is bi-$K_q$-invariant, satisfies $\omega(e) = 1$ and satisfies 
$$\omega * \varphi = (\omega * \varphi)(e) \omega$$
for every $\varphi \in \sH(G_q, K_q)$. 
\end{definition}
From Lemma \ref{LemmaHeckeAlgebraGenerator} we get the following.
\begin{corollary}
\label{CorollarySphericalFunctionsConvolutionEquation}
A bi-$K_q$-invariant function $\omega : G_q \rightarrow \C$ satisfying $\omega(e) = 1$ is $K_q$-spherical if $\omega * \chi_o = \alpha \omega$ for some $\alpha \in \C$. 
\end{corollary}
We will now construct for each $\alpha \in \C$ a spherical function $\omega$, providing a converse to the Corollary above. For the statements of the subsequent results we will parametrize $\alpha \in \C$ by taking $\lambda \in \C$ such that
\begin{align*}
\alpha = \alpha_{\lambda} = 2\sqrt{q}\cos_q(\lambda) \, . 
\end{align*}
Clearly, every $z_o \in \C$ can be written in this way, and the mapping $\lambda \mapsto \alpha_{\lambda}$ is bijective on $\{\lambda \in \C : 0 \leq \Re(\lambda) < \tau \}$ where again $\tau = 2\pi/\log(q)$. For the construction of spherical functions, we first note that a convolution $\varphi * \chi_o$ with $\varphi : G_q \rightarrow \C$ locally constant can be written using the Cartan decomposition as
\begin{align*}
(\varphi * \chi_o)(g) = (q + 1)\int_{K_q}\int_{K_q} \varphi(gk_1s_o^{-1}k_2) dm_{K_q}(k_1) dm_{K_q}(k_2) \, . 
\end{align*}
In particular, if $f : \qplusonetree \rightarrow \C$ and we introduce the function $\varphi_f(g) = f(g.o)$ then $\varphi_f$ is right-$K_q$-invariant and
\begin{align*}
(\varphi_f * \chi_o)(g) = (q + 1) \int_{K_q} \varphi_f(gk_1s_o^{-1}) dm_{K_q}(k_1) = \sum_{y \in \partial B_1(g.o)} f(y) \, . 
\end{align*}
This leads us to the following definition.

\begin{definition}
The \emph{adjacency operator} $\Delta_q : \Borelinfty(\qplusonetree) \rightarrow \Borelinfty(\qplusonetree)$ is 
\begin{align*}
\Delta_q f(x) = \sum_{y \in \partial B_1(x)} f(y) \, . 
\end{align*}
\end{definition}
By Corollary \ref{CorollarySphericalFunctionsConvolutionEquation}, the $K_q$-spherical functions on $\qplusonetree$ are in one-to-one correspondence with the $K_q$-invariant eigenfunctions of the adjacency operator $\Delta_q$. If $f$ is such an eigenfunction, then the explict relation is that $\omega(g) = f(g.o)$ is $K_q$-spherical. Writing out the equation $\Delta_q f = \alpha_{\lambda} f$ in terms of $\omega$ we get the following system of recursion:
\begin{align}
\label{EqSystemOfRecursionForSphericalFunctions}
\begin{cases}
\omega(e) = 1 \\
\omega(s_o) = \tfrac{\alpha_{\lambda}}{q + 1}\\
q\omega(s_o^{n+1}) - \alpha_{\lambda}\omega(s_o^{n}) + \omega(s_o^{n-1}) = 0  \, , \quad n \geq 2 \, .
\end{cases}
\end{align}
This system can be solved by elementary means. The solutions are the following, and a proof can be found in \cite[Proposition 2.4]{Figa-TalamancaNebbiaBook} using a slightly different parameterisation.
\begin{proposition}
\label{PropositionExplicitFormulaForTheSphericalFunctions}
Let $\tau = 2\pi/\log(q)$ and suppose that $\omega_{\lambda}$ is a $K_q$-spherical function such that $\omega_{\lambda} * \chi_o = \alpha_{\lambda}\omega_{\lambda}$. Then
\begin{align*}
\omega_{\lambda}(s_o^{n}) = \begin{dcases} c_q(\lambda) q^{-(1/2 - i\lambda)n} + c_q(-\lambda) q^{-(1/2 + i\lambda)n} &\mbox{ if } \lambda \in \C \backslash \tfrac{\tau}{2} \Z \\ 
\Big( 1 + \frac{q^{1/2} - q^{-1/2}}{q^{1/2} + q^{-1/2}} n\Big) (-1)^{j n} q^{-n/2} &\mbox{ if } \lambda = \tfrac{\tau}{2} j \in \tfrac{\tau}{2} \Z \end{dcases}
\end{align*}
where  
\begin{align*}
c_q(\lambda) = \frac{1}{q^{1/2} + q^{-1/2}} \frac{q^{1/2 + i\lambda} - q^{-1/2 - i\lambda}}{q^{i\lambda} - q^{-i\lambda}} = \frac{\sin_q(\lambda - \frac{i}{2})}{2 \cosh_q(\frac{1}{2}) \sin_q(\lambda)}
\end{align*}
is the Harish-Chandra $c$-function for $G_q$. 
\end{proposition}
Note that $\omega_{-\lambda} = \omega_{\lambda}$ and $\omega_{\lambda + j\tau/2} = \omega_{\lambda}$ for every $j \in \Z$, so we only need to restrict our attention to the family of spherical functions $\omega_{\lambda}$ such that $0 \leq \Re(\lambda) \leq \tau/2$. We also note that the modulus squared of the Harish-Chandra $c$-function can be written as
\begin{align}
\label{EqcFunctionModulusSquared}
|c_q(\lambda)|^2 = \frac{\sinh_q^2(\tfrac{1}{2}) + \sin_q^2(\lambda)}{4\cosh^2_q(\frac{1}{2})\sin_q^2(\lambda)} = \frac{1}{4\cosh^2_q(\frac{1}{2})} \Big( \frac{\sinh^2_q(\frac{1}{2})}{\sin^2_q(\lambda)} + 1 \Big)
\end{align}
when $\lambda \in (0, \tfrac{\tau}{2})$, which will be useful for later. Also, note that $|c_q(\lambda)|^2$ is minimized at $\lambda = \tau/4$ and that this minimum is positive.

\subsection{Spherical representations and positive-definite spherical functions}
\label{Spherical representations and positive-definite spherical functions}

Recall that a continuous function $\omega : G_q \rightarrow \C$ is \emph{positive-definite} if for all compactly supported locally constant $\varphi : G_q \rightarrow \C$ we have
\begin{align*}
\int_{G_q} (\varphi^* * \varphi)(g) \omega(g) dm_{G_q}(g) \geq 0 \, . 
\end{align*}
The positive-definite $K_q$-spherical functions on $G_q$ have a special connection to the unitary dual of $G_q$ that we summarize now. For a more detailed exposition, see \cite[Chapter 2]{Figa-TalamancaNebbiaBook}.

An irreducible unitary $G_q$-representation $(\pi, \cH)$ is \emph{$K_q$-spherical} if the subspace $\cH^{K_q}$ of $\pi(K_q)$-invariant vectors is at most of complex dimension 1. Given a $K_q$-invariant cyclic unit vector $v \in \cH^{K_q}$, one shows that the matrix coefficient
\begin{align*}
\omega(g) = \langle \pi(g)v, v \rangle_{\cH}
\end{align*}
is $K_q$-spherical and positive-definite. Conversely, given a positive-definite $K_q$-spherical function $\omega$, the classical GNS-construction yields a $K_q$-spherical irreducible unitary $G_q$-representation $(\pi_{\omega}, \cH_{\omega})$ with $\omega$ as the corresponding matrix coefficient for some cyclic $K_q$-invariant vector. To summarize, we have informally described Theorem 5.3 in \cite{Figa-TalamancaNebbiaBook}.
\begin{lemma}
A $K_q$-spherical function $\omega$ on $G_q$ is positive-definite if and only if it is the matrix coefficient of a $K_q$-spherical irreducible unitary $G_q$-representation $(\pi, \cH, v)$. 
\end{lemma}
This means that the $K_q$-spherical unitary dual $\hat{G}_q^{K_q}$ of equivalence classes of $K_q$-spherical irreducible unitary $G_q$-representations can be parameterized by a subset of spherical parameters $\lambda \in \C$ for which the $K_q$-spherical functions $\omega_{\lambda}$ are positive-definite. In \cite[p.53-56]{Figa-TalamancaNebbiaBook}, it is also shown that the positive-definite spherical functions are precisely the functions $\omega_{\lambda}$ such that
\begin{align*}
\lambda \in \Lambda_q := i[0, 1/2] \cup (0, \tau/2) \cup (\tau/2 + i[0, 1/2]) \, .
\end{align*}
In terms of adjacency operator eigenvalues, $\alpha_{\lambda} \in [-(q + 1), (q + 1)]$. In particular, the constant function $\omega_{i/2}(g) = 1$ and the sign function $\omega_{(\tau + i)/2}(g) = (-1)^{d(g.o, o)}$ are positive-definite, corresponding to the trivial and sign representation respectively.

\subsection{The spherical transform}
\label{The spherical transform}

\begin{definition}
The \emph{spherical transform} of an $m_{G_q}$-integrable function $\varphi : G_q \rightarrow \C$ is the function $\hat{\varphi} : \C \rightarrow \C$ given by
\begin{align*}
\hat{\varphi}(\lambda) = (\omega_{\lambda} * \varphi)(e) = \int_{G_q} \varphi(g) \omega_{\lambda}(g) dm_{G_q}(g) \, . 
\end{align*}
\end{definition}
If $\varphi$ is bi-$K_q$-invariant, for example when $\varphi \in \sH(G_q, K_q)$, then we can write $\hat{\varphi}$ using Cartan coordinates as
\begin{align*}
\hat{\varphi}(\lambda) = \sum_{n = 0}^{\infty} \varphi(s_o^n) \omega_{\lambda}(s_o^n) |\partial B_n (o)| = \varphi(0) + \frac{2 \cosh_q(\frac{1}{2})}{q^{1/2}} \sum_{n = 1}^{\infty} \varphi(s_o^n) \omega_{\lambda}(s_o^n) q^n \, . 
\end{align*}
\begin{lemma}[Spherical Plancherel formula]
\label{LemmaSphericalPlancherelFormula}
For every $\varphi \in \sH(G_q, K_q)$,
\begin{align*}
\int_{G_q} |\varphi(g)|^2 dm_{G_q}(g) = \frac{q^{1/2}}{2\cosh_q(\frac{1}{2})} \frac{1}{\tau} \int_0^{\tau/2} |\hat{\varphi}(\lambda)|^2 \frac{d\lambda}{|c_q(\lambda)|^2} \, .
\end{align*}
\end{lemma}
Again, in Cartan coordinates the Plancherel formula reads as
\begin{align*}
|\varphi(0)|^2 + \frac{2 \cosh_q(\frac{1}{2})}{q^{1/2}} \sum_{n = 1}^{\infty} |\varphi(s_o^n)|^2 q^n = \frac{q^{1/2}}{2\cosh_q(\frac{1}{2})} \frac{1}{\tau}  \int_0^{\tau/2} |\hat{\varphi}(\lambda)|^2 \frac{d\lambda}{|c_q(\lambda)|^2} \, .
\end{align*}

\subsection{The spherical transform of the indicator function of a ball}
\label{The spherical transform of the indicator function of a ball}

Computing the spherical transform of the indicator function $\chi_{B_r}$ of the metric ball $B_r$ of radius $r \in \Z_{\geq 0}$ will be crucial for studying the (sub-oscillatory) number variance of a point process $\mu$, which will be introduced in Subsection \ref{Conditional number variances}. 

\begin{lemma}
\label{LemmaFTofIndicatorFunction}
The spherical transform of the indicator function $\chi_{B_r}$ over the centered ball of radius $r \in \Z_{\geq 0}$ is
\begin{align*}
\hat{\chi}_{B_r}(\lambda) = 
\begin{dcases}
\frac{\sin_q(\lambda(r + 1)) + q^{-1/2}\sin_q(\lambda r)}{\sin_q(\lambda)} \, q^{r/2}  &\mbox{ if } \lambda \in \C \backslash \tfrac{\tau}{2} \Z \\
\Big( 1 + (1 + q^{-1/2}) r \Big) q^{r/2} &\mbox{ if } \lambda = \tfrac{\tau j}{2}, \, j \in 2\Z \\
\Big( 1 + (1 - q^{-1/2}) r \Big) (-1)^r q^{r/2} &\mbox{ if } \lambda = \tfrac{\tau j}{2}, \, j \in 2\Z + 1 \, . 
\end{dcases}
\end{align*}
\end{lemma}

\begin{remark}
Using the definition of the spherical transform, we can express the spherical functions as 
\begin{align*}
\omega_{\lambda}(s_o^n) = \frac{1}{1 + q^{-1}} \Big(\hat{\chi}_{B_r}(\lambda) - \hat{\chi}_{B_{r - 1}}(\lambda)\Big) q^{-r} \, . 
\end{align*}
By Lemma \ref{LemmaFTofIndicatorFunction}, we get
\begin{align*} 
\omega_{\lambda}(s_o^n) = 
\begin{dcases}
\frac{\sin_q((r + 1)\lambda) - q^{-1}\sin_q((r - 1)\lambda)}{(1 + q^{-1})\sin_q(\lambda)} \, q^{-r/2}  &\mbox{ if } \lambda \in \C \backslash \tfrac{\tau}{2} \Z \\
\Big( 1 + \frac{1 - q^{-1}}{1 + q^{-1}} r \Big) q^{-r/2} &\mbox{ if } \lambda = \tfrac{\tau j}{2}, \, j \in 2\Z \\
\Big( 1 + \frac{1 - q^{-1}}{1 + q^{-1}} r \Big) (-1)^r q^{-r/2} &\mbox{ if } \lambda = \tfrac{\tau j}{2}, \, j \in 2\Z + 1 \, ,
\end{dcases}
\end{align*}
which is an alternative formula to that in Proposition \ref{PropositionExplicitFormulaForTheSphericalFunctions}. 
\end{remark}

\begin{lemma}
\label{LemmaLinearTimesExponentialSum}
Let $a, z \in \C$ and $r \in \Z_{\geq 1}$. Then
\begin{align*}
\sum_{n = 1}^r (1 + an)z^n = \frac{az}{z - 1} r z^r + \frac{z(z - 1 - a)}{(z - 1)^2} (z^r  - 1) \, . 
\end{align*}
\end{lemma}
\begin{proof}
Straightforward computation.
\end{proof}

\begin{proof}[Proof of Lemma \ref{LemmaFTofIndicatorFunction}]
When $\lambda \in \C \backslash \frac{\tau}{2} \Z$, we have 
\begin{align*}
\hat{\chi}_{B_r}(\lambda) &= 1 + \frac{2 \cosh_q(\frac{1}{2})}{q^{1/2}} \Big( c_q(\lambda) \sum_{n = 1}^r q^{(\frac{1}{2} + i\lambda)n} + c_q(-\lambda) \sum_{n = 1}^r q^{(\frac{1}{2} - i\lambda)n} \Big) \\
&= 1 + \frac{2 \cosh_q(\frac{1}{2})}{q^{1/2}} \Big( c_q(\lambda) q^{\frac{1}{2} + i\lambda} \frac{q^{(\frac{1}{2} + i\lambda)r} - 1}{q^{\frac{1}{2} + i\lambda} - 1} + c_q(-\lambda) q^{\frac{1}{2} - i\lambda} \frac{q^{(\frac{1}{2} - i\lambda)r} - 1}{q^{\frac{1}{2} - i\lambda} - 1} \Big) \, . 
\end{align*}
Note that
\begin{align*}
c_q(\lambda) q^{\frac{1}{2} + i\lambda} \frac{q^{(\frac{1}{2} + i\lambda)r} - 1}{q^{\frac{1}{2} + i\lambda} - 1} &= \frac{(q^{\frac{1}{2} + i\lambda} + 1)(q^{(\frac{1}{2} + i\lambda)r} - 1)}{4\cosh_q(\frac{1}{2}) \sinh_q(i\lambda)}  = \frac{q^{(\frac{1}{2} + i\lambda)(r + 1)} + q^{(\frac{1}{2} + i\lambda)r} - q^{\frac{1}{2} + i\lambda} - 1}{4\cosh_q(\frac{1}{2}) \sinh_q(i\lambda)}
\end{align*}
so that
\begin{align*}
c_q(\lambda) q^{\frac{1}{2} + i\lambda} \frac{q^{(\frac{1}{2} + i\lambda)r} - 1}{q^{\frac{1}{2} + i\lambda} - 1} &+ c_q(-\lambda) q^{\frac{1}{2} - i\lambda} \frac{q^{(\frac{1}{2} - i\lambda)r} - 1}{q^{\frac{1}{2} - i\lambda} - 1} = \\
&= \frac{q^{1/2}\sinh_q(i\lambda(r + 1)) + \sinh_q(i\lambda r)}{2 \cosh_q(\frac{1}{2}) \sinh_q(i\lambda)} q^{r/2} - \frac{q^{1/2}}{2 \cosh_q(\frac{1}{2})} \, . 
\end{align*}
Plugging this into the formula for $\hat{\chi}_{B_r}$ we get 
\begin{align*}
\hat{\chi}_{B_r}(\lambda) &= 1 + \frac{2 \cosh_q(\frac{1}{2})}{q^{1/2}} \Big( \frac{q^{1/2}\sinh_q(i\lambda(r + 1)) + \sinh_q(i\lambda r)}{2 \cosh_q(\frac{1}{2}) \sinh_q(i\lambda)} q^{r/2} - \frac{q^{1/2}}{2 \cosh_q(\frac{1}{2})} \Big) \\
&= \frac{\sinh_q(i\lambda(r + 1)) + q^{-1/2}\sinh_q(i\lambda r)}{\sinh_q(i\lambda)} q^{r/2} = \frac{\sin_q(\lambda(r + 1)) + q^{-1/2}\sin_q(\lambda r)}{\sin_q(\lambda)} q^{r/2} \, . 
\end{align*}

For the second case when $\lambda = \tau j /2$, $j \in \Z$ even, we have 
\begin{align*}
\hat{\chi}_{B_r}(\lambda) &= 1 + \frac{2 \cosh_q(\frac{1}{2})}{q^{1/2}} \sum_{n = 1}^r \Big(1 + \tanh_q\Big(\frac{1}{2}\Big) n\Big) q^{n/2}  
\end{align*}
We use Lemma \ref{LemmaLinearTimesExponentialSum} with $z = q^{1/2}$ and $a = \tanh_q(\frac{1}{2})$ to find that
\begin{align*}
\sum_{n = 1}^r \Big(1 + \tanh_q\Big(\frac{1}{2}\Big) n\Big) q^{n/2} = \frac{\tanh_q(\frac{1}{2}) q^{1/2}}{q^{1/2} - 1} r q^{r/2} + \frac{q^{1/2}(q^{1/2} - 1 - \tanh_q(\frac{1}{2}))}{(q^{1/2} - 1)^2} (q^{r/2} - 1) \, . 
\end{align*}
We rewrite the coefficients in the equation above as
\begin{align*}
\frac{\tanh_q(\frac{1}{2}) q^{1/2}}{q^{1/2} - 1} = \frac{q - 1}{(q^{1/2} + q^{-1/2})(q^{1/2} - 1)} = \frac{q^{1/2} + 1}{q^{1/2} + q^{-1/2}} = \frac{q^{1/4}\cosh_q(\frac{1}{4})}{\cosh_q(\frac{1}{2})} 
\end{align*}
and 
\begin{align*}
\frac{q^{1/2}(q^{1/2} - 1 - \tanh_q(\frac{1}{2}))}{(q^{1/2} - 1)^2} &= q^{1/2}\frac{(q^{1/2} + q^{-1/2})(q^{1/2} - 1) - (q^{1/2} - q^{-1/2})}{(q^{1/2} + q^{-1/2})(q^{1/2} - 1)^2} \\
&= \frac{q^{1/2}}{q^{1/2} + q^{-1/2}} = \frac{q^{1/2}}{2\cosh_q(\frac{1}{2})} \, . 
\end{align*}
Finally we get that
\begin{align*}
\hat{\chi}_{B_r}(\lambda) &= 1 + \frac{2 \cosh_q(\frac{1}{2})}{q^{1/2}} \Big( \frac{q^{1/4}\cosh_q(\frac{1}{4})}{\cosh_q(\frac{1}{2})} r q^{r/2} + \frac{q^{1/2}}{2\cosh_q(\frac{1}{2})} (q^{r/2} - 1) \Big) \\
&= \Big( 1 + \frac{2 \cosh_q(\frac{1}{4})}{q^{1/4}} r \Big) q^{r/2} = \Big( 1 + (1 + q^{-1/2}) r \Big) q^{r/2} \, . 
\end{align*}

For $\lambda = \tau j /2$, where $j \in \Z$ is odd, one makes use of Lemma \ref{LemmaLinearTimesExponentialSum} with $z = -q^{1/2}$ and $a = \tanh_q(\frac{1}{2})$ and perform the same computations as in the second case to obtain the desired formula. 
%
\end{proof}

\begin{remark}
The second and third formula in Lemma \ref{LemmaFTofIndicatorFunction} can be obtained from the first formula by taking the limit $\lambda \rightarrow 0, \tau/2$ respectively. This implies continuity of $\hat{\chi}_{B_r}$.
\end{remark}

For the proof of non-geometric hyperuniformity in Section \ref{Non-geometric hyperuniformity} we will need the following lemmata.

\begin{lemma}
\label{LemmaCosineAsymptoticAverge}
Let $\lambda \in \R \backslash \frac{\tau}{2} \Z$, $a \in \{1, 2\}$ and $b \in \C$. Then
\begin{align*}
\lim_{R \rightarrow +\infty} \frac{1}{R} \sum_{r = 0}^{R} \cos_q(\lambda(ar + b)) = 0 \, . 
\end{align*}
\end{lemma}
\begin{proof}
This is a simple consequence of the formula for geometric sums,
\begin{align*}
\sum_{r = 0}^{R} \cos_q(\lambda(ar + b)) = \Re\Big( q^{i\lambda b} \sum_{r = 0}^{R} q^{i\lambda a r} \Big) = \Re\Big( q^{i\lambda b} \frac{q^{i\lambda a (R + 1)} - 1}{q^{i \lambda a} - 1} \Big) \leq \frac{2 q^{|\lambda||b|}}{|q^{i \lambda a} - 1|} \, . 
\end{align*}
The latter is finite since $\lambda a$ is not a multiple of $\tau$ and so we are done.
\end{proof}

\begin{lemma}
\label{LemmaAsymptoticMeanOfIndicatorFTSquared}
Let $\lambda \in (0, \tfrac{\tau}{2})$. Then
\begin{align*}
\lim_{R \rightarrow +\infty} \frac{1}{R} \sum_{r = 0}^{R} \frac{|\hat{\chi}_{B_r}(\lambda)|^2}{q^r} = \frac{1 + 2q^{-1/2} \cos_q(\lambda) + q^{-1}}{2 \sin_q^2(\lambda)} \, . 
\end{align*}
Moreover, for every $\varepsilon > 0$ the convergence is uniform over all $\lambda \in [\varepsilon, \tfrac{\tau}{2} - \varepsilon]$.
\end{lemma}

\begin{proof}
By the formula for $\hat{\chi}_{B_r}$ in the first case of Lemma \ref{LemmaFTofIndicatorFunction}, we have 
\begin{align*}
\frac{1}{R} \sum_{r = 0}^{R} \frac{|\hat{\chi}_{B_r}(\lambda)|^2}{q^r} &= \frac{1}{R \sin_q^2(\lambda)} \sum_{r = 0}^{R} (\sin_q(\lambda(r + 1)) + q^{-1/2}\sin_q(\lambda r))^2 \, . 
\end{align*}
Using the trigonometric identities $2 \sin_q(\alpha) \sin_q(\beta) = \cos_q(\alpha - \beta) - \cos_q(\alpha + \beta)$ and $2\sin_q^2(\alpha) = 1 + \cos_q(2\alpha)$ for $\alpha, \beta \in \R$ we get that
\begin{align*}
(\sin_q(\lambda(r + 1)) + q^{-1/2}\sin_q(\lambda r))^2 &= \sin_q^2(\lambda(r + 1)) + q^{-1/2}(\cos_q(\lambda) - \cos_q(\lambda(2 r + 1))) \\
&\qquad\qquad\qquad\quad\,\, + q^{-1}\sin_q^2(\lambda r) \\
&= \frac{1 + 2q^{-1/2} \cos_q(\lambda) + q^{-1}}{2} - \frac{1}{q^{1/2}}\cos_q(\lambda(2 r + 1)) \\
&\quad\, - \frac{1}{2} \cos_q(2\lambda(r + 1)) - \frac{1}{2q} \cos_q(2\lambda r) \, . 
\end{align*}
If we sum the latter expression over $r$ from $0$ to $R$ and divide by $R \sin_q^2(\lambda)$, then Lemma \ref{LemmaCosineAsymptoticAverge} tells us that
\begin{align*}
\frac{1}{R} \sum_{r = 0}^{R} \frac{|\hat{\chi}_{B_r}(\lambda)|^2}{q^r} = \, &\frac{1 + 2q^{-1/2} \cos_q(\lambda) + q^{-1}}{2 \sin_q^2(\lambda)}  - \frac{1}{q^{1/2} R \sin_q^2(\lambda)} \sum_{r = 0}^{R} \cos_q(\lambda(2 r + 1))\\
&- \frac{1}{2 R \sin_q^2(\lambda)} \sum_{r = 0}^{R} \cos_q(2\lambda(r + 1)) - \frac{1}{2 q R \sin_q^2(\lambda)} \sum_{r = 0}^{R} \cos_q(2 \lambda r) \\
& \longrightarrow \frac{1 + 2q^{-1/2} \cos_q(\lambda) + q^{-1}}{2 \sin_q^2(\lambda)} 
\end{align*}
as $R \rightarrow +\infty$. Uniform convergence for $\lambda \in [\varepsilon, \tfrac{\tau}{2} - \varepsilon]$ follows from compactness and the continuity of $\hat{\chi}_{B_r}(\lambda)$ in $\lambda$.
\end{proof}

\section{Invariant point processes, diffraction and number variances}
\label{Invariant point processes, diffraction and number variances}

We define invariant locally square-integrable point processes, followed by diffraction measures and conditional number variances in Subsections \ref{Invariant point processes}, \ref{Diffraction measures} and \ref{Conditional number variances} respectively. We conclude the Section by defining the unit intensity invariant Poisson point process and invariant random lattice orbits.
\subsection{Invariant point processes}
\label{Invariant point processes}
Let $\Radonplus(\qplusonetree)$ denote the space of locally finite positive measures on $\qplusonetree$ with the topology of pointwise convergence, that is, $p_n \rightarrow p$ in $\Radonplus(\qplusonetree)$ if $p_n(\{x\}) \rightarrow p(\{x\})$ for all $x \in \qplusonetree$. The automorphism group $G_q$ acts on $\Radonplus(\qplusonetree)$ by push-forward of measures. We write $p(x) := p(\{x\})$ and note that every $p \in \Radonplus(\qplusonetree)$ can be written as 
\begin{align*}
p = \sum_{x \in \qplusonetree} p(x) \delta_x  
\end{align*}
where $\delta_x$ denotes the unit point mass at $x \in \qplusonetree$, which canonically identifies the space $\Radonplus(\qplusonetree)$ with the space positive functions $\qplusonetree \rightarrow \R_{\geq 0}$ under pointwise convergence. To every finitely supported function $f \in \Borelbndinfty(\qplusonetree)$ we associate a \emph{linear statistic} $\bS f : \Radonplus(\qplusonetree) \rightarrow \C$ by
\begin{align*}
\bS f (p) = \sum_{x \in \qplusonetree} p(x) f(x) \, . 
\end{align*}
Since every pair of distinct vertices in $\qplusonetree$ are of at least distance $1$ from each other, all subsets of $\qplusonetree$ are uniformly discrete, so if the function $x \mapsto p(x)$ is uniformly bounded then the linear statistic $\bS f$ applied to $p$ is bounded by
\begin{align*}
|\bS f(p)| \leq \max_{x \in \qplusonetree} p(\{x\}) \, \norm{f}_{\infty} \, . 
\end{align*}
By a \emph{point process} on $\qplusonetree$ we mean a Borel probability measure $\mu$ on $\Radonplus(\qplusonetree)$. Such a point process $\mu$ is \emph{simple} if it assigns full measure to the subspace
\begin{align*}
\Radonplus^{\rS}(\qplusonetree) = \Big\{ p \in \Radonplus(\qplusonetree) \, | \, p(x) \in \{0, 1\} \,\,\, \forall \, x \in \qplusonetree  \Big\} \, . 
\end{align*}
A point process $\mu$ is 
\begin{itemize}
    \item \emph{invariant} if $g_*\mu = \mu$ for all $g \in \Gqplusone$.
    \item \emph{locally $k$-integrable}, $k \in \N$, if for every $x \in \qplusonetree$,
    $$ \int_{\Radonplus(\qplusonetree)} p(x)^k d\mu(p) < + \infty \, . $$
    In terms of linear statistics, this means that $\bS\chi_B \in L^k(\Radonplus(\qplusonetree), \mu)$ for all finite $B \subset V_{\qplusonetree}$. Note that a locally $k$-integrable random measure $\mu$ is locally $k'$-integrable for every $1 \leq k' \leq k$. If $\bS\chi_B \in L^{\infty}(\Radonplus(\qplusonetree), \mu)$ for all finite $B$, then we say that $\mu$ is \emph{locally bounded}.
    \item \emph{translation bounded} if it is invariant and locally bounded.
\end{itemize}
Some observations that follow are the following.
\begin{itemize}
    \item Every simple point process $\mu$ in $\qplusonetree$ is locally bounded since $\bS\chi_B(p) \leq |B|$ for $\mu$-almost every $p \in \Radonplus(\qplusonetree)$. In particular, every invariant simple point process is translation bounded.
    \item If $\mu$ is an invariant locally integrable point process, then there is a constant $\iota_{\mu} > 0$, the \emph{intensity} of $\mu$, such that
    \begin{align*}
    \Emu(\bS f) = \int_{\Radonplus(\qplusonetree)} \bS f(p) d\mu(p) = \iota_{\mu} \sum_{x \in \qplusonetree} f(x) \, . 
    \end{align*}
    \item If $\mu$ is an invariant locally square-integrable point process, then the variance of a linear statistic $\bS f$ with respect to $\mu$ is
    \begin{align*}
    \Var_{\mu}(\bS f) &= \Emu(|\bS f - \Emu(\bS f)|^2) = \Emu(|\bS f|^2) - |\Emu(\bS f)|^2 \\
    &= \int_{\Radonplus(\qplusonetree)} \Big| \sum_{x \in \qplusonetree} p(x) f(x)\Big|^2 d\mu(p) - \iota_{\mu}^2 \Big| \sum_{x \in \qplusonetree} f(x) \Big|^2 \, . 
    \end{align*}
    \item If $\mu$ is an invariant locally square-integrable point process, then the action of $G_q$ on $\Radonplus(\qplusonetree)$ yields a \emph{unitary Koopman representation} $\pi_{\mu} : G_q \rightarrow \sU(L^2(\Radonplus(\qplusonetree), \mu))$,
    \begin{align*}
    \pi_{\mu}(g)\bS f(p) = \bS f(g^{-1}_*p) \, . 
    \end{align*}
    The diffraction measure, defined in the next Subsection, records the decomposition of the $K_q$-invariant linear statistics into $K_q$-spherical unitary irreducible subrepresentations of $\pi_{\mu}$. 
\end{itemize}

\textbf{Assumption}: We will from now on assume that every point process to be invariant and locally square-integrable, unless stated otherwise.

\subsection{Diffraction measures}
\label{Diffraction measures}

Recall the subset $\Lambda_q \subset \C$ of spherical parameters corresponding to the positive-definite spherical functions, 
\begin{align*}
\Lambda_q = i[0, 1/2] \cup (0, \tau/2) \cup (\tau/2 + i[0, 1/2]) \, . 
\end{align*}
In an upcoming paper together with M. Björklund we will prove the existence and uniqueness of the diffraction measure of an invariant locally square-integrable point process on any locally compact second countable commutative space associated with a Gelfand pair. In the special case of regular trees we have the following formulation.  
\begin{theorem}
\label{ThmDiffractionMeasure}
Let $\mu$ be an invariant locally square-integrable point process on $\qplusonetree$. Then there is a unique positive finite Borel measure $\sigma_{\mu}$ on $\Lambda_q \subset \C$ such that for every $f \in \Borelbndinfty(\qplusonetree)$,
\begin{align*}
\int_{\Radonplus(\qplusonetree)} |\bS f (p)|^2 d\mu(p) = \int_{\Lambda_q} |\hat{\varphi}_f(\lambda)|^2 d\sigma_{\mu}(\lambda) \, . 
\end{align*}
Moreover, $\sigma_{\mu}(\{i/2\}) = \iota_{\mu}^2$. 
\end{theorem}
We refer to $\sigma_{\mu}$ as the \emph{diffraction measure} of $\mu$. We split $\sigma_{\mu}$ into three measures $\sigma_{\mu}^{(c)}, \sigma_{\mu}^{(p)}, \sigma_{\mu}^{(c_{\mathrm{sgn}})}$ by restricting to the respective intervals $i[0, 1/2), (0, \tau/2), \tau/2 + i[0, 1/2)$. In other words, for every $\varphi \in \sH(G_q, K_q)$,
\begin{align*}
\int_{\Lambda_q} \hat{\varphi}(\lambda) d\sigma_{\mu}(\lambda) = \,\, &\iota_{\mu}^2 |\hat{\varphi}(\tfrac{i}{2})|^2 + \int_0^{1/2} |\hat{\varphi}(i\lambda)|^2 d\sigma_{\mu}^{(c)}(\lambda)  +  \int_{0}^{\tau/2} |\hat{\varphi}(\lambda)|^2 d\sigma_{\mu}^{(p)}(\lambda) \\
&+ \int_0^{1/2} |\hat{\varphi}(\tfrac{\tau}{2} + i\lambda)|^2 d\sigma_{\mu}^{(c_{\mathrm{sgn}})}(\lambda) + \sigma_{\mu}(\{\tfrac{\tau + i}{2}\}) |\hat{\varphi}(\tfrac{\tau + i}{2})|^2 \, . 
\end{align*}
We refer to $\sigma_{\mu}^{(p)}, \sigma_{\mu}^{(c)}, \sigma_{\mu}^{(c_{\mathrm{sgn}})}$ as the \emph{principal-, complementary-} and \emph{signed complementary diffraction measure} respectively. We note that the spherical parameters $\lambda = 0, \tau/2$ correspond to principal series representations, but for our purposes it will be convenient to include them in the support of the complementary diffraction measures.

\subsection{Conditional number variances}
\label{Conditional number variances}

Let $\mu$ be an invariant locally square-integrable point process on $\qplusonetree$. The canonical variance of a linear statistic $\bS f$ is
\begin{align*}
\Varmu(\bS f) &= \int_{\Radonplus(\qplusonetree)} |\bS f(p)|^2 d\mu(p) - \iota_{\mu}^2 \Big|\sum_{x \in \qplusonetree} f(x) \Big|^2 \, .  
\end{align*}
Using the fact that
\begin{align*}
\sum_{x \in \qplusonetree} f(x) = \hat{\varphi}_f(\tfrac{i}{2})
\end{align*}
then Theorem \ref{ThmDiffractionMeasure} yields
\begin{align*}
\Varmu(\bS f) = \int_{\Lambda_q \backslash \{\frac{i}{2}\}} |\hat{\varphi}_f(\lambda)|^2 d\sigma_{\mu}(\lambda) \, .  
\end{align*}
In other words, the variance of a linear statistic is obtained as the second moment of the projected statistic in the orthogonal complement of the subspace of constant functions. This is however not the only way to obtain a notion of ''variance'', and in fact, for regular trees this is not always the preferred choice of variance.  
\begin{definition}[Spectrally conditioned variance]
Let $\mu$ be an invariant locally square-integrable point process and $Q \subset \Lambda_q$ be a Borel measurable subset. The \emph{$Q$-conditional variance} of a linear statistic $\bS f$, $f \in \Borelbndinfty(\qplusonetree)$, is 
\begin{align*}
\Var_{\mu}(\bS f \, | \, Q) = \int_{\Lambda_q \backslash Q} |\hat{\varphi}_f(\lambda)|^2 d\sigma_{\mu}(\lambda) \, . 
\end{align*}
\end{definition}
As mentioned in the Introduction, we will pay special attention to the cases $Q = \{\frac{i}{2}\}$ and $Q = \{\frac{i}{2}, \frac{\tau + i}{2}\}$, and we define number variances for these two cases.
\begin{definition}[Spectrally conditioned number variance]
Let $\mu$ and $Q$ be as in the previous definition. The \emph{$Q$-conditional number variance} of a finite subset $B \subset \qplusonetree$ is
\begin{align*}
\NV_{\mu}(B \,|\, Q) = \Var_{\mu}(\bS \chi_{B} \,|\, Q) \, . 
\end{align*}
Moreover, we define
\begin{enumerate}
    \item the \emph{number variance} $\NV_{\mu}(B) = \NV_{\mu}(B \,|\, \{\frac{i}{2}\})$,
    \item the \emph{sub-oscillatory number variance} $\NV_{\mu}^*(B) = \NV_{\mu}(B \,|\, \{\frac{i}{2}, \frac{\tau + i}{2}\})$.
\end{enumerate}
\end{definition}
The sub-oscillatory number variance of the metric ball $B_r$ can be written in terms of the diffraction measure $\sigma_{\mu}$ as
\small\begin{align*}
\NV_{\mu}^*(B_r) = \int_0^{1/2} |\hat{\chi}_{B_r}(i\lambda)|^2 d\sigma_{\mu}^{(c)}(\lambda)  +  \int_{0}^{\tau/2} |\hat{\chi}_{B_r}(\lambda)|^2 d\sigma_{\mu}^{(p)}(\lambda) + \int_0^{1/2} |\hat{\chi}_{B_r}(\tfrac{\tau}{2} + i\lambda)|^2 d\sigma_{\mu}^{(c_{\mathrm{sgn}})}(\lambda) .
\end{align*}\normalsize

\begin{example}[Invariant Poisson point process]
\label{ExampleInvariantPoissonPointProcess}
The unit intensity invariant Poisson point process on $\qplusonetree$ is the point process $\mu_{\Poi}$ defined implicitly by requiring that
\begin{enumerate}
    \item for every $k \in \Z_{\geq 0}$ and every finite subset $B \subset \qplusonetree$, the linear statistic $\bS\chi_B$ is Poisson distributed with intensity $|B|$ with respect to $\mu_{\Poi}$. 
    \item for every collection $B_1, \dots, B_m \subset \qplusonetree$ of disjoint finite subsets, the linear statistics $\bS\chi_{B_1}, \dots, \bS\chi_{B_m}$ are $\mu_{\Poi}$-independent.
\end{enumerate}
One shows that $\mu_{\Poi}$ is a simple point process and that it is unique up to equivalence, see \cite[Theorem 3.6 + Prop. 3.2]{Penrose2017LecturesOT}. If $f \in \Borelbndinfty(\qplusonetree)$, then the variance of the linear statistic $\bS f$ is 
\begin{align*}
\Var_{\mu_{\Poi}}(\bS f) &= \sum_{x, y \in \qplusonetree} f(x) \overline{f(y)} \Cov_{\mu_{\Poi}}(\bS\chi_{\{x\}}, \bS\chi_{\{y\}}) \\
&= \sum_{x \in \qplusonetree} |f(x)|^2 \Var_{\mu_{\Poi}}(\bS\chi_{\{x\}}) = \sum_{x \in \qplusonetree} |f(x)|^2 \, . 
\end{align*}
By the spherical Plancherel formula in Lemma \ref{LemmaSphericalPlancherelFormula} we see that the diffraction measure $\sigma_{\Poi} := \sigma_{\mu_{\Poi}}$ is given by the Plancherel measure,
\begin{align*}
\int_{\Lambda_q} |\varphi_f(\lambda)|^2 d\sigma_{\Poi}(\lambda) = \sum_{x \in \qplusonetree} |f(x)|^2 = \frac{q^{1/2}}{2\cosh_q(\frac{1}{2})} \frac{1}{\tau} \int_0^{\tau/2} |\varphi_f(\lambda)|^2 \frac{d\lambda}{|c_q(\lambda)|^2} \, . 
\end{align*}
In other words, $d\sigma_{\Poi}^{(c)}(\lambda) = d\sigma_{\Poi}^{(c_{\sgn})}(\lambda) = 0$ and $d\sigma_{\Poi}^{(p)}(\lambda) = \frac{q^{1/2}}{2\cosh_q(\frac{1}{2})} \frac{1}{\tau} |c_q(\lambda)|^{-2} d\lambda$.
\end{example}

\begin{example}[Random lattice orbits]
\label{ExampleRandomLatticeOrbits}
Let $\Gamma < G_q$ be a lattice. The \emph{random lattice orbit} $\mu_{\Gamma}$ is the invariant point process whose moments are given by
\begin{align*}
\int_{\Radonplus(\qplusonetree)} \bS f(p)^k d\mu_{\Gamma}(p) = \int_{\Gamma \backslash G_q} \Big( \sum_{\gamma \in \Gamma} f((\gamma g)^{-1}.o) \Big)^k dm_{\Gamma \backslash G_q}(\Gamma g) \, , 
\end{align*}
where $m_{\Gamma \backslash G_q}$ denotes the unique invariant probability measure on $\Gamma \backslash G_q$. In probabilistic terms, the invariant random lattice orbit is $g^{-1}\Gamma.o$, where $g$ is a uniformly distributed random automorphism in the fundamental domain of $\Gamma \backslash G_q$ in $G_q$. We compute the diffraction measure of $\mu_{\Gamma}$ when $\Gamma$ is the fundamental group of a finite simple connected $(q+1)$-regular graph in Section \ref{Hyperuniformity of fundamental groups of finite regular graphs}.
\end{example}

\section{Non-geometric hyperuniformity}
\label{Non-geometric hyperuniformity}

We will in this Section prove Theorem \ref{Theorem1.2} via the following analogue for regular trees of the inequality proved by Beck in \cite[p. 4]{BeckChenIrregularitiesOfDistribution}.

\begin{proposition}
\label{PropositionNonGeometricHyperuniformity}
Let $\mu$ be an invariant locally square-integrable point process on $\qplusonetree$ with diffraction measure $\sigma_{\mu}$. Then
\begin{enumerate}
    \item if $\sigma_{\mu}^{(c)} \neq 0$ or $\sigma_{\mu}^{(c_{\sgn})} \neq 0$,
    \begin{align*}
        \frac{\NVmu^*(r)}{q^r} \geq (1 + q^{-1/2})^2 \int_0^{1/2} \frac{\sinh_q^2(r\lambda)}{\sinh_q^2(\lambda)} d\sigma_{\mu}^{(c)}(\lambda) + (1 - q^{-1/2})^2 \int_0^{1/2} \frac{\sinh_q^2(r\lambda)}{\sinh_q^2(\lambda)} d\sigma_{\mu}^{(c_{\mathrm{sgn}})}(\lambda) \, .  
    \end{align*}
    \item if $\sigma_{\mu}^{(p)} \neq  0$, then for every $\varepsilon > 0$ there is an $R_o = R_o(q, \varepsilon) > 0$ such that for every $R \geq R_o$,
    \begin{align*}
        \frac{1}{R} \sum_{r = 0}^R \frac{\NVmu^*(r)}{q^r} \geq \frac{2(q + 1)^2}{q(q^{1/2} + 1)^2} \int_{\varepsilon}^{\tau/2 - \varepsilon} |c_q(\lambda)|^2 d\sigma_{\mu}^{(p)}(\lambda)\, . 
    \end{align*}
\end{enumerate}
\end{proposition}

\begin{proof}
For (1) we have the lower bound 
\begin{align*}
\frac{\NVmu^*(r)}{q^r} \geq \int_0^{1/2} \frac{|\hat{\chi}_{B_r}(i\lambda)|^2}{q^r} d\sigma_{\mu}^{(c)}(\lambda) + \int_0^{1/2} \frac{|\hat{\chi}_{B_r}(\tfrac{\tau}{2} + i\lambda)|^2}{q^r} d\sigma_{\mu}^{(c_{\mathrm{sgn}})}(\lambda) \, .
\end{align*}
We then bound $|\hat{\chi}_{B_r}|^2$ from below in each case using the explicit formulas in Lemma \ref{LemmaFTofIndicatorFunction} by
\begin{align*}
\frac{|\hat{\chi}_{B_r}(i\lambda)|^2}{q^r} = \frac{(\sinh_q((r + 1)\lambda) + q^{-1/2}\sinh_q(r\lambda))^2}{\sinh_q^2(\lambda)} \geq (1 + q^{-1/2})^2 \frac{\sinh_q^2(r\lambda)}{\sinh_q^2(\lambda)}
\end{align*}
and 
\begin{align*}
\frac{|\hat{\chi}_{B_r}(\tfrac{\tau}{2} + i\lambda)|^2}{q^r} &=  \frac{((-1)^{r + 1}\sinh_q((r + 1)\lambda) + q^{-1/2}(-1)^r\sinh_q(r\lambda))^2}{(-\sinh_q(\lambda))^2} \\
&=  \frac{(\sinh_q((r + 1)\lambda) - q^{-1/2}\sinh_q(r\lambda))^2}{\sinh_q^2(\lambda)} \geq (1 - q^{-1/2})^2 \frac{\sinh_q^2(r\lambda)}{\sinh_q^2(\lambda)}
\end{align*}
to get the desired result.

For (2) we have the lower bound 
\begin{align*}
        \frac{1}{R} \sum_{r = 0}^R \frac{\NVmu^*(r)}{q^r} \geq \int_{\varepsilon}^{\tau/2 - \varepsilon} \Big( \frac{1}{R} \sum_{r = 0}^R \frac{|\hat{\chi}_{B_r}(\lambda)|^2}{q^r} \Big) d\sigma_{\mu}^{(p)}(\lambda)\, . 
\end{align*}
By Lemma \ref{LemmaAsymptoticMeanOfIndicatorFTSquared} we get that
\begin{align*}
        \lim_{R \rightarrow +\infty} \int_{\varepsilon}^{\tau/2 - \varepsilon} \Big( \frac{1}{R} \sum_{r = 0}^R \frac{|\hat{\chi}_{B_r}(\lambda)|^2}{q^r} \Big) d\sigma_{\mu}^{(p)}(\lambda) = \int_{\varepsilon}^{\tau/2 - \varepsilon} \frac{1 + 2q^{-1/2}\cos_q(\lambda) + q^{-1}}{2\sin_q^2(\lambda)} d\sigma_{\mu}^{(p)}(\lambda) \, . 
\end{align*}
Lastly, using the formula for $|c_q(\lambda)|^2$ in Equation \ref{EqcFunctionModulusSquared},
\begin{align*}
\frac{1 + 2q^{-1/2}\cos_q(\lambda) + q^{-1}}{2\sin_q^2(\lambda)} &= \frac{2\cosh^2_q(\tfrac{1}{2})(1 + 2q^{-1/2}\cos_q(\lambda) + q^{-1})}{\sinh_q^2(\tfrac{1}{2}) + \sin_q^2(\lambda)} |c_q(\lambda)|^2 \\
&\geq 2\coth^2_q(\tfrac{1}{2})(1 - q^{-1/2})^2 |c_q(\lambda)|^2 = \frac{2(q + 1)^2}{q(q^{1/2} + 1)^2} |c_q(\lambda)|^2 \, . 
\end{align*}
Gathering everything together, we get 
\begin{align*}
        \lim_{R \rightarrow +\infty} \frac{1}{R} \sum_{r = 0}^R \frac{\NVmu^*(r)}{q^r} \geq  \frac{2(q + 1)^2}{q(q^{1/2} + 1)^2} \int_{\varepsilon}^{\tau/2 - \varepsilon} |c_q(\lambda)|^2 d\sigma_{\mu}^{(p)}(\lambda)\, . 
    \end{align*}
\end{proof}

\begin{proof}[Proof of Theorem \ref{Theorem1.2}]
To justify Equation \ref{EqNonGeometricHyperuniformity} in Theorem \ref{Theorem1.2}, we use item (2) in Proposition \ref{PropositionNonGeometricHyperuniformity} to find that 
\begin{align*}
\limsup_{r \rightarrow +\infty} \frac{\NVmu^*(r)}{|B_r|} \geq  \limsup_{r \rightarrow +\infty} \frac{q - 1}{(q + 1)R} \sum_{r = 0}^R \frac{\NVmu^*(r)}{q^r}  \geq \frac{2(q^2 - 1)}{q(q^{1/2} + 1)^2} \int_{\varepsilon}^{\tau/2 - \varepsilon} |c_q(\lambda)|^2 d\sigma_{\mu}^{(p)}(\lambda)
\end{align*}
for any $\varepsilon > 0$. Since $|c_q(\lambda)|^2 > 0$ for all $\lambda \in (0, \tau/2)$ with positive minimum achieved at $\lambda = \tau/4$, seen from Equation \ref{EqcFunctionModulusSquared}, we may take $\varepsilon > 0$ such that $\sigma_{\mu}^{(p)}((\varepsilon, \tau/2 - \varepsilon)) > 0$ to find that
\begin{align*}
\limsup_{r \rightarrow +\infty} \frac{\NVmu^*(r)}{|B_r|} &\geq \frac{2(q^2 - 1)}{q(q^{1/2} + 1)^2} \int_{\varepsilon}^{\tau/2 - \varepsilon} |c_q(\lambda)|^2 d\sigma_{\mu}^{(p)}(\lambda) \\
&\geq \frac{2(q^2 - 1)}{q(q^{1/2} + 1)^2} |c_q(\tfrac{\tau}{4})|^2 \sigma_{\mu}^{(p)}((\varepsilon, \tau/2 - \varepsilon)) > 0 , . 
\end{align*}
To prove item (1), we use item (1) of Proposition \ref{PropositionNonGeometricHyperuniformity} to see that
\begin{align*}
\frac{\NVmu^*(r)}{q^r} &\geq (1 + q^{-1/2})^2 \int_0^{1/2} \frac{\sinh_q^2(r\lambda)}{\sinh_q^2(\lambda)} d\sigma_{\mu}^{(c)}(\lambda) + (1 - q^{-1/2})^2 \int_0^{1/2} \frac{\sinh_q^2(r\lambda)}{\sinh_q^2(\lambda)} d\sigma_{\mu}^{(c_{\mathrm{sgn}})}(\lambda) \\
&\geq (1 + q^{-1/2})^2 r^2 \sigma_{\mu}^{(c)}(\{0\}) + (1 - q^{-1/2})^2 r^2 \sigma_{\mu}^{(c_{\sgn})}(\{\tfrac{\tau}{2}\}) \\
&\geq (1 - q^{-1/2})^2 r^2 \sigma_{\mu}(\{0, \tfrac{\tau}{2}\})  
\end{align*}
for every $r \geq 0$. From this we get
\begin{align*}
\liminf_{r \rightarrow +\infty} \frac{\NVmu^*(r)}{r^2|B_r|} \geq \frac{(q - 1)(1 - q^{-1/2})^2}{q + 1} \sigma_{\mu}(\{0, \tfrac{\tau}{2}\}) > 0 \, . 
\end{align*}
Similarly, we prove item (3) by observing that 
\begin{align*}
\liminf_{r \rightarrow +\infty} \frac{\NVmu^*(r)}{q^{(1 + \delta)r}} &\geq  (1 + q^{-1/2})^2 \liminf_{r \rightarrow +\infty} \int_{\delta/2}^{1/2} \frac{\sinh_q^2(r\lambda)}{q^{\delta r} \sinh_q^2(\lambda)} d\sigma_{\mu}^{(c)}(\lambda) \\
&\quad + (1 - q^{-1/2})^2 \liminf_{r\rightarrow +\infty} \int_{\delta/2}^{1/2} \frac{\sinh_q^2(r\lambda)}{q^{\delta r}\sinh_q^2(\lambda)} d\sigma_{\mu}^{(c_{\mathrm{sgn}})}(\lambda) \\
&\geq (1 + q^{-1/2})^2\liminf_{r \rightarrow +\infty} \frac{\sinh_q^2(r\frac{\delta}{2})}{q^{\delta r} \sinh_q^2(\frac{\delta}{2})}\sigma_{\mu}^{(c)}([\tfrac{\delta}{2}, \tfrac{1}{2}))  \\
&\quad + (1 - q^{-1/2})^2 \liminf_{r \rightarrow +\infty} \frac{\sinh_q^2(r\frac{\delta}{2})}{q^{\delta r} \sinh_q^2(\frac{\delta}{2})}\sigma_{\mu}^{(c_{\sgn})}([\tfrac{\delta}{2}, \tfrac{1}{2})) \\
&\geq \frac{(1 - q^{-1/2})^2}{4\sinh_q^2(\frac{\delta}{2})} \sigma_{\mu}\big(i[\tfrac{\delta}{2}, \tfrac{1}{2}) \cup (\tfrac{\tau}{2} + i[\tfrac{\delta}{2}, \tfrac{1}{2}))\big)\, .
\end{align*}
Thus 
\begin{align*}
\liminf_{r \rightarrow +\infty} \frac{\NVmu^*(r)}{|B_r|^{1 + \delta}} & \geq \frac{(q - 1)^{1 + \delta}(1 - q^{-1/2})^2}{4(q + 1)^{1 + \delta}\sinh_q^2(\frac{\delta}{2})}  \sigma_{\mu}\big(i[\tfrac{\delta}{2}, \tfrac{1}{2}) \cup (\tfrac{\tau}{2} + i[\tfrac{\delta}{2}, \tfrac{1}{2}))\big) > 0 \, . 
\end{align*}
\end{proof}

\section{Atoms in the diffraction picture}
\label{Atoms in the diffraction picture}

Let $\mu$ be an invariant locally square-integrable point process on $\qplusonetree$. If $\lambda \in \Lambda_q$ is an atom for the diffraction measure $\sigma_{\mu}$, meaning $\sigma_{\mu}(\{\lambda\}) > 0$, then the lower limit of the number variance by the volume of the ball can be bounded from below by
\begin{align*}
\liminf_{r \rightarrow +\infty} \frac{\NVmu^*(B_r)}{|B_r|} \geq \sigma_{\mu}(\{ \lambda \}) \liminf_{r \rightarrow +\infty} \frac{|\hat{\chi}_{B_r}(\lambda)|^2}{|B_r|} \, . 
\end{align*}
We know from Theorem \ref{Theorem1.2} that if $\lambda \in i[0, 1/2) \cup (\tau/2 + i[0, 1/2))$ then this lower limit is infinite, so we restrict our attention to principal series spherical parameters $\lambda \in (0, \tau/2)$. For this we now have by Lemma \ref{LemmaFTofIndicatorFunction} the explicit lower bound 
\begin{align*}
\liminf_{r \rightarrow +\infty} \frac{\NVmu(B_r)}{|B_r|} \geq \frac{(q - 1)\sigma_{\mu}(\{ \lambda \})}{(q + 1)\sin_q^2(\lambda)}\liminf_{r \rightarrow +\infty} \Big(\sin_q((r + 1)\lambda) + q^{-1/2} \sin_q(r\lambda) \Big)^2  \, . 
\end{align*}
We will now show that this lower limit is positive whenever the diffraction measure admits a "rational" atom in the spectrum. The following Proposition is proved in the Appendix, and we to state it we recall the notation $\tau = 2\pi/\log(q)$ and $\sin_q(\lambda) = \sin(\log(q)\lambda)$. 

\begin{proposition}
\label{PropositionTrigonometricEquationsRationalSolutions}
Let $a, b \in \Z$ be coprime such that $0 < a < b$. Then the equation
\begin{align*}
\sin_q(\tfrac{(r + 1)a \tau}{2b}) + q^{-1/2}\sin_q(\tfrac{r a \tau}{2b}) = 0
\end{align*}
has a solution $0 \leq r \leq b-1$ if and only if 
\begin{itemize}
    \item $q = 2$ and $\tfrac{a}{b} \in \{\tfrac{3}{4}, \tfrac{5}{12}, \tfrac{11}{12}\}$, or 
    \item $q = 3$ and $\tfrac{a}{b} = \tfrac{5}{6}$.
\end{itemize}
\end{proposition}

\begin{proof}[Proof of Theorem \ref{Theorem1.3}]
Suppose that the diffraction measure $\sigma_{\mu}$ has an atom of the form $\lambda = \tfrac{a\tau}{2b}$ for some coprime $a, b \in \Z$ such that $0 < a < b$. Denote by $F_{a, b} : \Z \rightarrow \R_{\geq 0}$ the function
\begin{align*}
F_{a, b}(r) = \Big(\sin_q((r + 1)\tfrac{a\tau}{2b}) + q^{-1/2} \sin_q(r\tfrac{a\tau}{2b}) \Big)^2 \, , 
\end{align*}
so that 
\begin{align*}
\liminf_{r \rightarrow +\infty} \frac{\NVmu^*(B_r)}{|B_r|} \geq \frac{(q - 1)\sigma_{\mu}(\{ \tfrac{a\tau}{2b} \})}{(q + 1)\sin_q^2(\tfrac{a\tau}{2b})}\liminf_{r \rightarrow +\infty} F_{a, b}(r) \, . 
\end{align*} 
Note that $F_{a, b}$ is not only a $2b$-periodic function, but also $b$-periodic from the square in its definition. By $b$-periodicity of $F_{a, b}$, the lower limit on the right hand side vanishes if and only if there is an $0 \leq r^* \leq b - 1 $ such that $F_{a, b}(r^*) = 0$. By Proposition \ref{PropositionTrigonometricEquationsRationalSolutions}, such an $r^*$ exists if and only if 
$$(q, \tfrac{a}{b}) = (2, \tfrac{3}{4}), (2, \tfrac{5}{12}), (2, \tfrac{11}{12}) \mbox{ or } (3, \tfrac{5}{6}) \, . $$
Thus, if $(q, \tfrac{a}{b})$ fall outside of the mentioned cases then we have that
\begin{align*}
\liminf_{r \rightarrow +\infty} F_{a, b}(r) = \min_{0 \leq r \leq b} \Big(\sin_q((r + 1)\tfrac{a\tau}{2b}) + q^{-1/2} \sin_q(r\tfrac{a\tau}{2b}) \Big)^2 > 0 
\end{align*}
and so we're done.
\end{proof}

\section{Spectral hyperuniformity and stealth}
\label{Spectral hyperuniformity and stealth}

We recall the definition of spectral hyperuniformity and stealth from the Introduction and compute the average asymptotic of the sub-oscillatory number variance for stealth point processes. Recall from Example \ref{ExampleInvariantPoissonPointProcess} that the diffraction measure of the unit intensity invariant Poisson point process is 
\begin{align*}
d\sigma_{\Poi}(\lambda) = \delta_{i/2} + \chi_{(0, \tau/2)}(\lambda)\frac{q^{1/2}}{2\cosh_q(\frac{1}{2})} \frac{1}{\tau} |c_q(\lambda)|^{-2} d\lambda \, , 
\end{align*}
where the Lebesgue absolutely continuous part of the measure is the Plancherel measure for $G_q$. We take the stance that spectral hyperuniformity of an invariant locally square-integrable point process $\mu$ should refer to sub-Poissonian decay of the diffraction measure $\sigma_{\mu}$ near the Harish-Chandra $\Xi$-functions corresponding to $0, \tau/2 \in \Lambda_q$. To state the definition, we observe from Equation \ref{EqcFunctionModulusSquared} that 
\begin{align*}
   |c_q(\lambda)|^{-2} =  \frac{4\cosh^2_q(\frac{1}{2})\sin_q^2(\lambda)}{\sinh_q^2(\tfrac{1}{2}) + \sin_q^2(\lambda)} \asymp_q \begin{cases} \lambda^{2} &\mbox{ if } \lambda \approx 0 \\ (\tfrac{\tau}{2} - \lambda)^2  &\mbox{ if } \lambda \approx \tfrac{\tau}{2} \end{cases}
\end{align*}
for $\lambda$ close to $0$ and $\tau/2$ respectively, so that
\begin{align*}
\sigma_{\Poi}\Big((0, \varepsilon] \cup [ \tfrac{\tau}{2} - \varepsilon, \tfrac{\tau}{2})\Big) &= \frac{q^{1/2}}{2\cosh_q(\frac{1}{2})} \frac{1}{\tau} \Big( \int_0^{\varepsilon} |c_q(\lambda)|^{-2} d\lambda + \int_{\tau/2 - \varepsilon}^{\tau/2} |c_q(\lambda)|^{-2} d\lambda \Big) \\
&\asymp_{q} \int_0^{\varepsilon} \lambda^2 d\lambda + \int_{\tau/2 - \varepsilon}^{\tau/2} (\tfrac{\tau}{2} - \lambda)^2 d\lambda \asymp \varepsilon^3
\end{align*}
as $\varepsilon \rightarrow 0^+$. The defintion is the following.
\begin{definition}
Let $\mu$ be an invariant locally square-integrable point process on $\qplusonetree$ with diffraction measure $\sigma_{\mu}$. Then $\mu$ is
\begin{enumerate}
    \item \emph{spectrally hyperuniform} if $\sigma_{\mu}^{(c)} = \sigma_{\mu}^{(c_{\mathrm{osc}})} = 0$ and 
    \begin{align*}
    \limsup_{\varepsilon \rightarrow 0^+} \frac{\sigma_{\mu}^{(p)}((0, \varepsilon])}{\varepsilon^3} = \limsup_{\varepsilon \rightarrow 0^+} \frac{\sigma_{\mu}^{(p)}([\tau/2 - \varepsilon, \tau/2))}{\varepsilon^3} = 0 \, . 
    \end{align*}
    \item \emph{stealthy} if $\sigma_{\mu}^{(c)} = \sigma_{\mu}^{(c_{\mathrm{osc}})} = 0$ and there is an $\varepsilon_o > 0$ such that 
    \begin{align*}
    \sigma_{\mu}^{(p)}((0, \varepsilon_o]) = \sigma_{\mu}^{(p)}([\tau/2 - \varepsilon_o, \tau/2)) = 0 \, . 
    \end{align*}
\end{enumerate}
\end{definition}
In particular, a stealthy point process is spectrally hyperuniform. For stealthy point processes $\mu$ we can explicitly compute the average asymptotic of the number variance in terms of the diffraction measure $\sigma_{\mu}$.

\begin{lemma}
Let $\mu$ be a stealthy point process on $\qplusonetree$. Then
\begin{align*}
        \lim_{R \rightarrow +\infty} \frac{1}{R} \sum_{r = 0}^R \frac{\NVmu^*(r)}{q^r} = \frac{1}{2} \int_{0}^{\tau/2} \frac{1 + 2q^{-1/2} \cos_q(\lambda) + q^{-1}}{\sin_q^2(\lambda)} d\sigma_{\mu}^{(p)}(\lambda) \, . 
\end{align*}
\end{lemma}

\begin{proof}
Since $\mu$ is stealthy, there is an $\varepsilon_o > 0$ such that $\sigma_{\mu}^{(p)}((0, \varepsilon_o)) = \sigma_{\mu}^{(p)}((\tau/2 - \varepsilon_o, \tau/2)) = 0$. Thus 
\begin{align*}
       \frac{1}{R} \sum_{r = 0}^R \frac{\NVmu^*(r)}{q^r} = \int_{\varepsilon_o}^{\tau/2 - \varepsilon_o}  \Big( \frac{1}{R} \sum_{r = 0}^R \frac{|\hat{\chi}_{B_r}(\lambda)|^2}{q^r} \Big) d\sigma_{\mu}^{(p)}(\lambda) \, . 
\end{align*}
By Lemma \ref{LemmaAsymptoticMeanOfIndicatorFTSquared}, the limit of the above as as $R \rightarrow +\infty$ becomes
\begin{align*}
        \lim_{R \rightarrow +\infty} \int_{\varepsilon_o}^{\tau/2 - \varepsilon_o}  \Big( \frac{1}{R} \sum_{r = 0}^R \frac{|\hat{\chi}_{B_r}(\lambda)|^2}{q^r} \Big) d\sigma_{\mu}^{(p)}(\lambda) &= \int_{\varepsilon_o}^{\tau/2 - \varepsilon_o} \frac{1 + 2q^{-1/2} \cos_q(\lambda) + q^{-1}}{2 \sin_q^2(\lambda)} d\sigma_{\mu}^{(p)}(\lambda) \\
        &= \frac{1}{2} \int_{0}^{\tau/2} \frac{1 + 2q^{-1/2} \cos_q(\lambda) + q^{-1}}{\sin_q^2(\lambda)} d\sigma_{\mu}^{(p)}(\lambda) \, . 
\end{align*}
\end{proof}

\section{Hyperuniformity of fundamental groups of finite regular graphs}
\label{Hyperuniformity of fundamental groups of finite regular graphs}

We construct the universal covering tree of a finite simple connected regular graph in Subsection \ref{The universal covering tree of a simple connected regular graph} following \cite{ChihScullFundamentalGroupoidsOfGraphs}, along with the fundamental group of such a graph in Subsection \ref{The fundamental group}, realizing such graphs as double coset graphs for $G_q$. The diffraction measure of the associated random lattice orbits are computed in Subsectin \ref{Diffraction measures of fundamental groups} and finally we prove Theorem \ref{Theorem1.4} for the complete, regular bipartite and Petersen graph in Subsection \ref{Examples}. 

\subsection{The universal covering tree of a simple connected regular graph}
\label{The universal covering tree of a simple connected regular graph}

Let $\frak{X} = (V_{\frak{X}}, E_{\frak{X}})$ be a simple connected $(q + 1)$-regular graph and fix a vertex $u \in V_{\frak{X}}$.
Consider the set of $\bP_{\frak{X}}$ of paths of finite length in $\frak{X}$, in other words sequences $\xi = (x_0, x_1, \dots, x_n)$ such that $x_i \in V_{\frak{X}}$ and $(x_i, x_{i+1}) \in E_{\frak{X}}$ for all $i$. If $\xi, \xi' \in \bP_{\frak{X}}$ share the same final and initial vertex respectively, then we can consider the concatenated path $\xi.\xi' \in \bP_{\frak{X}}$, as well as reversed paths $\xi^{-1} \in \bP_{\frak{X}}$. If $\xi \in \bP_{\frak{X}}$ has \emph{backtracking}, meaning that there is an $i$ such that $x_i = x_{i+2}$, then a \emph{prune} of $\xi$ is a path
\begin{align*}
\xi' = (x_0, x_1, \dots , x_{i - 1}, x_i, x_{i + 3}, \dots , x_n ) \in \bP_{\frak{X}} \, . 
\end{align*}
Two paths $\xi, \xi' \in \bP_{\frak{X}}$ are \emph{prune-equivalent}, $\xi \sim_{p.e.} \xi'$, if there exists a sequence of paths $\xi = \xi_1, \xi_2, \dots, \xi_{m-1}, \xi_m = \xi'$ such that either $\xi_i$ is a prune of $\xi_{i+1}$ or $\xi_{i+1}$ is a prune of $\xi_i$ for all $i$. Note that the initial and final vertex of a path are preserved under prune-equivalence. A path $\xi$ without backtracking is called \emph{non-prunable}. One shows that prune-equivalence defines an equivalence relation on $\bP_{\frak{X}}$, and the following minimality property can be proved by induction on the number of backtrackings in a path, see \cite[Prop. 3.3]{ChihScullFundamentalGroupoidsOfGraphs}.
\begin{lemma}
\label{LemmaNonPrunableRepresentative}
Each equivalence class $[\xi]_{p.e.}$ of $\sim_{p.e.}$ has a unique non-prunable representative $\xi^{\min}$.
\end{lemma}
Note that if $\xi_1, \xi_2 \in \bP_{\frak{X}}$ share the final and initial vertex respectively, then prune-equivalence commutes with concatenation in the sense that
\begin{align*}
\xi_1.\xi_2 \sim_{p.e.} \xi_1^{\min}.\xi_2 \sim_{p.e.} \xi_1.\xi_2^{\min} \sim_{p.e.} \xi_1^{\min}.\xi_2^{\min} \sim_{p.e.} (\xi_1.\xi_2)^{\min} \, . 
\end{align*}
The concatenation of two non-prunable paths are not necessarily non-prunable, so the last equivalence in the chain above might not be an equality of paths. 

The \emph{rooted pruned tree} $\bT_{\frak{X}}$ of the pointed graph $(\frak{X}, u)$ is the graph whose vertex set are the prune-equivalence classes of paths in $\bP_{\frak{X}}$ that start with $u$,
\begin{align*}
V_{\bT_{\frak{X}}} = \Big\{ [\xi]_{p.e.} \in \bP_{\frak{X}}/\sim_{p.e.} \, \Big| \, \xi = (u, x_1, ..., x_n) \in \bP_{\frak{X}} \Big\}
\end{align*}
and with edge set given by pairs of prune-equivalence classes such that one is the extension of the other by one neighbouring vertex in $\frak{X}$,
\begin{align*}
E_{\bT_{\frak{X}}} = \Big\{ ([\xi_1]_{p.e.}, [\xi_2]_{p.e.}) \in V_{\bT_{\frak{X}}} \times V_{\bT_{\frak{X}}} \, \Big| \, \exists \, x \in V_{\frak{X}} :  \xi_1^{\min} = (\xi_2^{\min}, x) \mbox{ or } \xi_2^{\min} = (\xi_1^{\min}, x) \Big\} \, . 
\end{align*}
Moreover, we fix the root vertex $\xi_u = [(u)]_{p.e.} \in V_{\bT_{\frak{X}}}$. Then $\bT_{\frak{X}}$ is simple, connected and $(q + 1)$-regular since $\frak{X}$ is and $\bT_{\frak{X}}$ does not contain any cycles, so we conclude that there is a graph isomorphism $\bT_{\frak{X}} \rightarrow \qplusonetree$ sending $\xi_u$ to $o$.

The rooted pruned tree $\bT_{\frak{X}}$ is the metric universal covering space of $\frak{X}$ whose covering map is the graph homomorphism $\pr_{\frak{X}} : \bT_{\frak{X}} \rightarrow \frak{X}$ given by sending $[\xi]_{p.e.} \in V_{\bT_{\frak{X}}}$ with non-prunable representative $\xi^{\min} = (u, x_1, \dots, x_n)$ to the vertex $x_n \in V_{\frak{X}}$. In particular, it sends the root vertex $\xi_u$ to the root vertex $u$.

\subsection{The fundamental group}
\label{The fundamental group}

Consider a simple connected $(q + 1)$-regular pointed graph $(\frak{X}, u)$ with universal covering tree $\bT_{\frak{X}}$ as in the previous subsection. The \emph{fundamental group} of $(\frak{X}, u)$ is the set of prune-equivalence classes of cycles in $\frak{X}$ based at $u$,
\begin{align*}
\pi_1(\frak{X}) = \pi_1(\frak{X}, u) = \Big\{ [\alpha]_{p.e.} \in V_{\bT_{\frak{X}}} \, \Big| \, \alpha \mbox{ ends with } u \Big\} 
\end{align*}
with identity $\xi_u = [(u)]_{p.e.}$, endowed with the operation of concatenation and inversion given by reversal of paths. Explicitly, if $\gamma = [\alpha]_{p.e.} = [(u, x_1, \dots, x_m, u)]_{p.e.}$ and $\gamma' = [\alpha']_{p.e.} = [(u, y_1, \dots, y_n, u)]_{p.e.}$ in $\pi_1(\frak{X})$ then
\begin{align*}
\gamma \gamma' = [\alpha.\alpha']_{p.e.} = [(u, x_1, \dots, x_m, u, y_1, \dots, y_n, u)]_{p.e.}  
\end{align*}
and 
\begin{align*}
\gamma^{-1} = [\alpha^{-1}]_{p.e.} = [(u, x_m, \dots, x_1, u)]_{p.e.} \, . 
\end{align*}
Moreover, $\gamma = [\alpha]_{p.e.} \in \pi_1(\frak{X})$ defines an automorphism of $\bT_{\frak{X}}$ by concatenation from the left, $\gamma.[\xi]_{p.e.} = [\alpha.\xi]_{p.e.}$. Note that this action is well-defined, for if $\alpha_1, \alpha_2$ are two representatives of $\gamma$ then by Lemma \ref{LemmaNonPrunableRepresentative},
\begin{align*}
\alpha_1.\xi \sim_{p.e.} \alpha_1^{\min}.\xi = \alpha_2^{\min}.\xi \sim_{p.e.} \alpha_2.\xi
\end{align*}
for all $\xi \in \bT_{\frak{X}}$.
\begin{lemma}
The action of $\pi_1(\frak{X})$ on $\bT_{\frak{X}}$ is free. In particular, $\pi_1(\frak{X}) < \Aut(\bT_{\frak{X}})$ is a discrete subgroup.
\end{lemma}
\begin{proof}
To prove freeness of the action, suppose that $\gamma = [\alpha]_{p.e.} \in \pi_1(\frak{X})$ such that $\gamma.[\xi]_{p.e.} = [\xi]_{p.e.}$ for some $[\xi]_{p.e.} \in \bT_{\frak{X}}$. By associativity of concatenation we then get
\begin{align*}
\alpha \sim_{p.e.} \alpha.\xi.\xi^{-1} \sim_{p.e.} \xi.\xi^{-1} \sim_{p.e.} (u) \, , 
\end{align*}
so that $\gamma = [\alpha]_{p.e.} = \xi_u$ is the identity. 

To prove that $\pi_1(\frak{X}) < \Aut(\bT_{\frak{X}})$ is a discrete subgroup, suppose without loss of generality that there is a sequence $\gamma_n \rightarrow \xi_u$ in $\pi_1(\frak{X})$. Since the stabilizer of $\xi_u \in \bT_{\frak{X}}$ is an open subgroup of $\Aut(\bT_{\frak{X}})$ we must have that $\gamma_n.\xi_u = \xi_u$ for sufficiently large $n$. Also, the action of $\pi_1(\frak{X})$ on $\bT_{\frak{X}}$ is free, so we must have that $\gamma_n = \xi_u$ for all sufficiently large $n$ and hence $\pi_1(\frak{X}) < \Aut(\bT_{\frak{X}})$ is discrete.
\end{proof}
Via a pointed graph isomorphism $i : (\bT_{\frak{X}}, \xi_u) \rightarrow (\qplusonetree, o)$ we get a discrete subgroup $\Gamma_{\frak{X}} = i^* \pi_1(\frak{X}) < G_q$, where $i^* : \Aut(\bT_{\frak{X}}) \rightarrow G_q$ is the group isomorphism given by $i^*g = i \circ g \circ i^{-1}$ for $g \in \Aut(\bT_{\frak{X}})$. Moreover, if $i_1, i_2 : (\bT_{\frak{X}}, \xi_u) \rightarrow (\qplusonetree, o)$ are two different pointed graph isomorphisms then $k := i_2 \circ i_1^{-1} \in K_q$ and the discrete subgroups $\Gamma_1 =i_1^* \pi_1(\frak{X}), \Gamma_2 =i_2^* \pi_1(\frak{X})$ are related by $\Gamma_{2} = k\Gamma_1k^{-1}$, so we may think of $\Gamma_{\frak{X}} < G_q$ as uniquely determined up to conjugation by elements of $K_q$.

This will now allow us to realize the graph $\frak{X}$ as a quotient of the tree $\qplusonetree = G_q/K_q$.

\begin{definition}[Quotient graphs]
Let $\frak{X}$ be a digraph and $\Gamma < \Aut(\frak{X})$ a subgroup acting freely on $\frak{X}$ from the left. The \emph{quotient digraph} $\Gamma \backslash \frak{X}$ is the digraph with vertex set
\begin{align*}
V_{\Gamma \backslash \frak{X}} = \Gamma \backslash V_{\frak{X}}
\end{align*}
and edge set
\begin{align*}
E_{\Gamma \backslash \frak{X}} = \Big\{ (\Gamma.x_1, \Gamma.x_2) \in \Gamma \backslash V_{\frak{X}} \times \Gamma \backslash V_{\frak{X}} \, \Big| \, \exists \, \gamma \in \Gamma : (x_1, \gamma.x_2) \in E_{\frak{X}}  \Big\} \, . 
\end{align*}
\end{definition}
With this notion we can now realize our simple connected $(q + 1)$-regular graph $\frak{X}$ as the quotient graph of the universal covering tree $\bT_{\frak{X}}$ by the fundamental group $\pi_1(\frak{X})$. 
\begin{proposition}
\label{PropositionGraphAsQuotientOfitsUniversalCoveringTree}
The pointed graph $(\frak{X}, u)$ is isomorphic to the pointed quotient graph $(\pi_1(\frak{X}) \backslash \bT_{\frak{X}}, \pi_1(\frak{X}).\xi_u)$.  
\end{proposition}
\begin{proof}
The canonical map $\pi_1(\frak{X}) \backslash \bT_{\frak{X}} \rightarrow \frak{X}$ sending $\pi(\frak{X}).[\xi]_{p.e.}$ to $\pr_{\frak{X}}([\xi]_{p.e.})$ is a surjective pointed graph homomorphism. To see that it is injective, note that if $\pr_{\frak{X}}([\xi_1]_{p.e.}) = \pr_{\frak{X}}([\xi_2]_{p.e.})$ then the concatinated class $\gamma := [\xi_1.\xi_2^{-1}]_{p.e.}$ defines an element of $\pi_1(\frak{X})$. Thus $\gamma.[\xi_2]_{p.e.} = [\xi_1]_{p.e.}$ and so $\pi_1(\frak{X}).[\xi_1]_{p.e.} = \pi_1(\frak{X}).[\xi_2]_{p.e.}$. Note that the element $\gamma$ we constructed is the unique element of $\pi_1(\frak{X})$ such that $\gamma.[\xi_2]_{p.e.} = [\xi_1]_{p.e.}$. From this fact it follows that the inverse map is also a graph homomorphism.
\end{proof}
By Proposition \ref{PropositionGraphAsQuotientOfitsUniversalCoveringTree} we can now canonically identify $\frak{X}$ with $\Gamma_{\frak{X}} \backslash G_q / K_q$ such that the root vertex $u \in \frak{X}$ corresponds to the trivial coset $\Gamma_{\frak{X}} e K_q \in \Gamma_{\frak{X}} \backslash G_q / K_q$. 
\begin{corollary}
Suppose that $\frak{X}$ is a finite simple connected $(q + 1)$-regular graph. Then the fundamental group $\Gamma_{\frak{X}} < G_q$ is a cocompact lattice. 
\end{corollary}

\subsection{Diffraction measures of fundamental groups}
\label{Diffraction measures of fundamental groups}

Let $\Gamma = \Gamma_{\frak{X}}$ for some finite simple connected $(q + 1)$-regular graph $\frak{X}$ and consider the random lattice orbit $\mu_{\frak{X}} := \mu_{\Gamma}$ in $\qplusonetree$. Denote by $\rho : G_q \rightarrow \sU(L^2(\Gamma \backslash G_q))$ the right-regular representation, so that the linear isomorphism
\begin{align*}
\cI_{\Gamma} : L^2(\Radonplus(\qplusonetree), \mu_{\Gamma}) \longrightarrow L^2(\Gamma \backslash G_q, m_{\Gamma \backslash G_q}) \, , \quad \cI_{\Gamma} (\bS f)(\Gamma g) = \bS f(\delta_{g^{-1}\Gamma.o})
\end{align*}
is equivariant in the sense that $\cI_{\Gamma}(\pi_{\mu_{\Gamma}}(g) \bS f) = \rho(g^{-1})\cI_{\Gamma}(\bS f)$ for all $g \in G_q$, where $\pi_{\mu_{\Gamma}}$ is the Koopman representation mentioned in Subsection \ref{Invariant point processes}. In this case we will be able to compute the diffraction measure $\sigma_{\frak{X}} := \sigma_{\mu_{\Gamma}}$ in terms of the spectrum of the normalized adjacency operator $\Delta_{\frak{X}}$ on $\frak{X}$, given by
\begin{align*}
\Delta_{\frak{X}}F(\Gamma g K_q) = \sum_{\Gamma g K_q \in  \partial B_1(\Gamma h K_q) } F(\Gamma h K_q) = (q + 1) \int_{K_q} F(\Gamma g k s_o^{-1} K_q) dm_{K_q}(k) \, . 
\end{align*}
\begin{lemma}
\label{LemmaLubotskyCorrespondence}
Let $\frak{X}$ be a finite simple connected $(q + 1)$-regular graph with fundamental group $\Gamma < G_q$ and let $\lambda \in \Lambda_q$. Then $\alpha_{\lambda} \in \C$ is an eigenvalue of $\Delta_{\frak{X}}$ if and only if there is a $K_q$-spherical irreducible sub-representation of $(\rho, L^2(\Gamma \backslash G_q))$ with $K_q$-spherical matrix coefficient $\omega_{\lambda}$. 
\end{lemma}
We follow \cite[Section 5.5]{LubotzkyBook} in the proof.
\begin{proof}
First, assume that there is an $F \in \ell^2(\frak{X})$ such that  $\Delta_{\frak{X}} F = \alpha_{\lambda} F$. Then we set $\tilde{F}(\Gamma g) = \norm{F}_2^{-1} F(\Gamma g K_q)$ and consider the bi-$K_q$-invariant matrix coefficient $\omega(g) = \langle \rho(g)\tilde{F}, \tilde{F} \rangle$ for $g \in G_q$. It suffices to prove that $\omega = \omega_{\lambda}$ for some $\lambda \in \Lambda_q$. Fubini yields
\begin{align*}
(\omega * \chi_o)(g) &= (q + 1) \int_{K_q} \omega(gks_o^{-1}) dm_{K_q}(k) \\
&= \frac{1}{\norm{F}_2} \int_{\Gamma \backslash G_q} \Big( (q + 1)\int_{K_q} F(\Gamma h g k s_o^{-1} K_q) dm_{K_q}(k) \Big) \overline{\tilde{F}(\Gamma h)} dm_{\Gamma \backslash G_q}(\Gamma h) \, , 
\end{align*}
and since 
\begin{align*}
(q + 1)\int_{K_q} F(\Gamma h g k s_o^{-1} K_q) dm_{K_q}(k) = \Delta_{\frak{X}} F(\Gamma h g K_q) = \alpha_{\lambda} F(\Gamma h g K_q)
\end{align*}
by assumption, we get
\begin{align*}
(\omega * \chi_o)(g) = \frac{\alpha_{\lambda}}{\norm{F}_2} \int_{\Gamma \backslash G_q}  F(\Gamma h g K_q) \overline{\tilde{F}(\Gamma h)} dm_{\Gamma \backslash G_q}(\Gamma h) = \alpha_{\lambda} \omega(g) \, . 
\end{align*}
By Corollary \ref{CorollarySphericalFunctionsConvolutionEquation} this means precisely that $\omega(g) = \omega_{\lambda}(g)$ for all $g \in G_q$. 

Conversely, suppose that there is a $\tilde{F} \in L^2(\Gamma \backslash G_q)^{K_q}$ such that $\norm{\tilde{F}}_2 = 1$ with $\langle \rho(g)\tilde{F}, \tilde{F} \rangle = \omega_{\lambda}(g)$ for all $g \in G_q$. Then for $F(\Gamma g K_q) := \tilde{F}(\Gamma g)$ we have 
\begin{align*}
\Delta_{\frak{X}}F(\Gamma g K_q) = (q + 1) \int_{K_q} \tilde{F}(\Gamma g k s_o^{-1}) dm_{K_q}(k) = (q + 1) \int_{K_q} \rho(k s_o^{-1}) \tilde{F}(\Gamma g) dm_{K_q}(k) \, . 
\end{align*}
Since $\rho(ks_o^{-1})\tilde{F}$ remains in the same $K_q$-spherical irreducible sub-representation as $\tilde{F}$ for every $k \in K_q$ and 
\begin{align*}
g \longmapsto \int_{K_q} \rho(k s_o^{-1}) \tilde{F}(\Gamma g) dm_{K_q}(k)
\end{align*}
is $K_q$-invariant function in $L^2(\Gamma \backslash G_q)^{K_q}$, we must have that
\begin{align*}
\int_{K_q} \rho(k s_o^{-1}) \tilde{F}(\Gamma g) dm_{K_q}(k) = \alpha_{\tilde{F}} \, \tilde{F}(\Gamma g) 
\end{align*}
for some $\alpha_{\tilde{F}} \in \C$ and all $g \in G_q$. Taking the inner product with $\tilde{F}$ on both sides of this equation we see that
\begin{align*}
\alpha_{\tilde{F}} = \int_{K_q} \langle \rho(ks_o^{-1})\tilde{F}, \tilde{F} \rangle dm_{K_q}(k) &= \int_{K_q} \langle \rho(s_o^{-1})\tilde{F}, \rho(k^{-1})\tilde{F} \rangle dm_{K_q}(k) \\
&= \langle \rho(s_o^{-1})\tilde{F}, \tilde{F} \rangle = \omega_{\lambda}(s_o^{-1}) = \omega_{\lambda}(s_o) = \frac{\alpha_{\lambda}}{q + 1} \, ,
\end{align*}
where we used Lemma \ref{LemmaInversionofKbyKCosets} in the second to last equality and Equation \ref{EqSystemOfRecursionForSphericalFunctions} for the last equality. We conclude that 
\begin{align*}
\Delta_{\frak{X}}F(\Gamma g K_q) = \alpha_{\lambda} \tilde{F}(\Gamma g) = \alpha_{\lambda} F(\Gamma g K_q)
\end{align*}
as desired.
\end{proof}
If we denote by $\Lambda_{\frak{X}} \subset \Lambda_q$ the spherical parameters $\lambda$ corresponding to eigenvalues $\alpha_{\lambda} \neq q+1$ of $\Delta_{\frak{X}}$, then Lemma \ref{LemmaLubotskyCorrespondence} tells us that 
\begin{align*}
\sigma_{\frak{X}} = \frac{1}{|V_{\frak{X}}|^2}\delta_{i/2} + \sum_{\lambda \in \Lambda_{\frak{X}}} m_{\frak{X}}(\lambda) \delta_{\lambda}
\end{align*}
where the $m_{\frak{X}}(\lambda) > 0$ are constants related to the multiplicity of the corresponding spherical representation. From the definition of spectral hyperuniformity and stealth in Section \ref{Spectral hyperuniformity and stealth} and Theorem \ref{Theorem1.2} we have the following.
\begin{corollary}
The random lattice orbit $\mu_{\frak{X}}$ of the fundamental group $\Gamma_{\frak{X}}$ of a finite simple connected $(q + 1)$-regular graph $\frak{X}$ is stealthy if and only if $\Lambda_{\frak{X}} \subset (0, \tfrac{\tau}{2})$. If the latter does not hold, then $\mu_{\frak{X}}$ is geometrically hyperfluctuating.
\end{corollary}
%

\subsection{Examples}
\label{Examples}

\textbf{Complete graphs}: Let $\frak{K}_{q + 1}$ denote the complete graph on $q + 2$ vertices, whose edge set is $E_{\frak{K}_{q+1}} = \{ (v, w) \in V_{\frak{K}_{q + 1}} \times V_{\frak{K}_{q + 1}} \, | \,  v \neq w\}$. It is simple, connected and $(q + 1)$-regular, and its fundamental group $\Gamma_{\frak{K}_{q + 1}} < G_q$ is isomorphic to the free group on $\binom{q+1}{2}$ generators given by the prune-equivalence classes of cycles $[(u, v, w, u)]_{p.e.}$ for any pair $v, w \in V_{\frak{K}_{q + 1}}$ up to reordering. Moreover, the characteristic polynomial of the adjacency operator $\Delta_{\frak{K}_{q+1}}$ can be computed to be
\begin{align*}
\det(\Delta_{\frak{K}_{q+1}} - \alpha_{\lambda} \Id) = (-1)^{q + 1} (\alpha_{\lambda} + 1)^{q}(\alpha_{\lambda} - q) \, ,
\end{align*}
whose zeros are $\alpha_{\lambda} = q, -1$ with multiplicity $1$ and $q$ respectively. These correspond to $\lambda = \tfrac{i}{2}$ and $\lambda = \arccos_q(-\tfrac{1}{2\sqrt{q}}) =: \lambda_1(q)$ with respective multiplicities, so the diffraction measure of $\mu_{\frak{K}_{q + 1}}$ is
\begin{align*}
\sigma_{\frak{K}_{q+1}} = \frac{1}{(q + 2)^2} \delta_{i/2} + m_{\frak{K}_{q + 1}}(\lambda_1(q)) \delta_{\lambda_1(q)}  \, . 
\end{align*}
\begin{proof}[Proof of (1) in Theorem \ref{Theorem1.4}]
By Corollary \ref{CorollaryQuadraticIrrationalTrigonometricNumbers} we have that 
$$\lambda_1(q) = \arccos_q(-\tfrac{1}{2\sqrt{q}}) = \log(q)^{-1}\arccos(-\tfrac{1}{2\sqrt{q}})$$
is an irrational multiple of $\tau/2 = \pi/\log(q)$. Thus we can find a sequence $r_k \in \Z_{\geq 1}$ with $r_k \rightarrow +\infty$ such that $r_k\lambda_1(q) \rightarrow \lambda_1(q)$ in $\R/\tfrac{\tau}{2}\Z$, so that 
\small\begin{align*}
\lim_{k \rightarrow +\infty} \frac{\NV_{\frak{K}_{q+1}}^*(B_{r_k})}{|B_{r_k}|} &= \frac{(q - 1)m_{\frak{K}_{q + 1}}(\lambda_1(q))}{(q + 1)\sin_q^2(\lambda_1(q))} \lim_{k \rightarrow +\infty} \Big(\sin_q((r_k + 1)\lambda_1(q)) + q^{-1/2} \sin_q(r_k\lambda_1(q)) \Big)^2  \\
&= \frac{(q - 1)m_{\frak{K}_{q + 1}}(\lambda_1(q))}{(q + 1)\sin_q^2(\lambda_1(q))} \Big(\sin_q(2\lambda_1(q)) + q^{-1/2} \sin_q(\lambda_1(q)) \Big)^2 \\
&= \frac{(q - 1)m_{\frak{K}_{q + 1}}(\lambda_1(q))}{(q + 1)\sin_q^2(\lambda_1(q))} \Big(\sin_q(2\lambda_1(q)) - 2 \cos_q(\lambda_1(q))\sin_q(\lambda_1(q)) \Big)^2 = 0 \, . 
\end{align*}
\end{proof}\normalsize
\textbf{Regular complete bipartite graphs}: Let $\frak{B}_{q + 1} = \frak{K}_{q + 1, q + 1}$ denote the complete symmetric bipartite graph on $2(q + 1)$ vertices. Denote by 
$$V_{\frak{B}_{q+1}} = V_{\frak{B}_{q+1}}^1 \cup \, V_{\frak{B}_{q+1}}^2$$
the partition of the vertex set into $(q + 1) + (q + 1)$ vertices such that there are only edges between $V_{\frak{B}_{q+1}}^1$ and $V_{\frak{B}_{q+1}}^2$, so that
\begin{align*}
E_{\frak{B}_{q+1}} = \Big\{ (v, w) \in V_{\frak{B}_{q+1}} \times V_{\frak{B}_{q+1}} \,\Big|\, v \in V_{\frak{B}_{q+1}}^1 \mbox{ and } w \in V_{\frak{B}_{q+1}}^2 \Big\} \, . 
\end{align*}
Let us fix a vertex $u \in V_{\frak{B}_{q+1}}^1$ and a vertex $u' \in V_{\frak{B}_{q+1}}^2$. The fundamental group $\Gamma_{\frak{B}_{q+1}}$ is then isomorphic to the free group on $q^2$ generators with generating classes of the form $[(u, u', v, w, u)]_{p.e.}$ for $v \in V_{\frak{B}_{q+1}}^1 \backslash \{u\}$ and $w \in V_{\frak{B}_{q+1}}^2 \backslash \{u'\}$.

The adjacency operator for $\frak{B}_{q+1}$ can be written in the standard basis as the block matrix
\[
\Delta_{\frak{B}_{q + 1}} = 
\begin{pmatrix}
    0 & \Delta_{\frak{K}_{q+1}} + \Id \\
    \Delta_{\frak{K}_{q+1}} + \Id & 0
\end{pmatrix} \, .
\]
The characteristic polynomial of $\Delta_{\frak{B}_{q + 1}}$ can then be computed as 
\begin{align*}
\det(\Delta_{\frak{B}_{q+1}} - \alpha_{\lambda} \Id) &= -\det((\Delta_{\frak{K}_{q+1}} + \Id)^2 - \alpha_{\lambda}^2 \Id) \\
&= -\det(\Delta_{\frak{K}_{q+1}} - (\alpha_{\lambda} - 1) \Id) \det(\Delta_{\frak{K}_{q+1}} + (\alpha_{\lambda} + 1) \Id)  \\
&= (-1)^q (\alpha_{\lambda} - (q + 1)) \alpha_{\lambda}^{2q}(\alpha_{\lambda} + (q + 1)) \, ,
\end{align*}
with roots $\alpha_{\lambda} = q + 1, 0, -(q + 1)$ and respective multiplicities $1, 2q, 1$. The corresponding spherical parameters are $\lambda = \tfrac{i}{2}, \tfrac{\tau}{4}, \tfrac{\tau + i}{2}$ with the same multiplicities, so the diffraction measure of the random lattice orbit $\mu_{\frak{B}_{q+1}}$ is 
\begin{align*}
\sigma_{\frak{B}_{q+1}} = \frac{1}{4(q + 1)^2} \delta_{i/2} + m_{\frak{B}_{q + 1}}(\tfrac{\tau}{4}) \delta_{\tau/4} + m_{\frak{B}_{q+1}}(\tfrac{\tau + i}{2})\delta_{(\tau + i)/2}  \, . 
\end{align*}
\begin{proof}[Proof of (2) in Theorem \ref{Theorem1.4}]
Since $\tau/4$ is an atom of $\sigma_{\Gamma_{\frak{B}_{q+1}}}$, it follows from Theorem \ref{Theorem1.3} that 
\begin{align*}
\liminf_{r \rightarrow +\infty} \frac{\NV_{\frak{B}_{q + 1}}^*(B_r)}{|B_r|} > 0 \, . 
\end{align*}
\end{proof}

\textbf{The $3$-regular Petersen graph}: The Petersen graph $\frak{P}_3$ is a simple connected $3$-regular graph on $10$ vertices $u = v_1, v_2, \dots v_{10}$ which can be thought of as a joining of two copies of the cyclic graph on $5$ vertices. Explicitly, its adjacency operator is given in the standard basis by the block matrix
\begin{align*}
\Delta_{\frak{P}_3} = \begin{pmatrix} C_1 & \Id \\ \Id & C_2 \end{pmatrix}
\end{align*}
where
\begin{align*}
C_1 = \begin{pmatrix} 0 & 1 & 0 & 0 & 1 \\ 1 & 0 & 1 & 0 & 0 \\ 0 & 1 & 0 & 1 & 0 \\ 0 & 0 & 1 & 0 & 1 \\ 1 & 0 & 0 & 1 & 0 \end{pmatrix}
\, , \quad \quad 
C_2 = \begin{pmatrix} 0 & 0 & 1 & 1 & 0 \\ 0 & 0 & 0 & 1 & 1 \\ 1 & 0 & 0 & 0 & 1 \\ 1 & 1 & 0 & 0 & 0 \\ 0 & 1 & 1 & 0 & 0 \end{pmatrix} \, . 
\end{align*} 
The fundamental group $\Gamma_{\frak{P}_3}$ is isomorphic to the free group on $6$ generators, given for example by the presentation
\begin{align*}
\Gamma_{\frak{P}_3} \cong \Big\langle &[(u, v_6, v_8, v_{10}, v_5, u)]_{p.e.}, [(u, v_6, v_9, v_{7}, v_2, u)]_{p.e.}, [(u, v_6, v_8, v_{3}, v_2, u)]_{p.e.}, \\
&[(u, v_6, v_9, v_{4}, v_5, u)]_{p.e.}, [(u, v_6, v_8, v_{10}, v_7, v_9, v_6, u)]_{p.e.}, [(u, v_6, v_8, v_{3}, v_4, v_9, v_6, u)]_{p.e.} \Big\rangle \, . 
\end{align*}

One computes the characteristic polynomial of $\Delta_{\frak{P}_3}$ to be 
\begin{align*}
\det(\Delta_{\frak{P}_3} - \alpha_{\lambda}\Id) = (\alpha_{\lambda} - 3)(\alpha_{\lambda} - 1)^5(\alpha_{\lambda} + 2)^4 
\end{align*}
whose roots have the corresponding spherical parameters 
$$\lambda_0 = i/2, \quad \lambda_1 = \arccos_2(\tfrac{1}{2\sqrt{2}}), \quad  \lambda_2 = 3\tau/8 \, . $$
The diffraction measure of the random lattice orbit $\mu_{\frak{P}_3}$ can then be written as
\begin{align*}
\sigma_{\frak{P}_{3}} = \frac{1}{100} \delta_{i/2} + m_{\frak{P}_{3}}(\lambda_1) \delta_{\lambda_1} + m_{\frak{P}_3}(\tfrac{3\tau}{8})\delta_{3\tau/8} \, . 
\end{align*}
\begin{proof}[Proof of (3) in Theorem \ref{Theorem1.4}]
From the above formula for the diffraction measure of $\Gamma_{\frak{P}_3}$ we see that 
\begin{align*}
\liminf_{r \rightarrow +\infty} \frac{\NV_{\frak{P}_{3}}^*(B_r)}{|B_r|} = 0 \, . 
\end{align*}
if there is a sequence $r_k \in \Z_{\geq 1}$ with $r_k \rightarrow +\infty$ such that 
\begin{align*}
\lim_{k \rightarrow +\infty}  \sin_2((r_k + 1)\lambda_1) + 2^{-1/2}\sin_2(r_k\lambda_1) = 0
\end{align*}
where again $\lambda_1 = \arccos_2(\tfrac{1}{2\sqrt{2}})$, and such that 
\begin{align*}
\lim_{k \rightarrow +\infty}  \sin_2((r_k + 1)\tfrac{3\tau}{8}) + 2^{-1/2}\sin_2(r_k\tfrac{3\tau}{8}) = 0 \, . 
\end{align*}
To find such a sequence, we first note that for every $n \in \Z$ we have 
\begin{align*}
\sin_2((2 + 4n + 1)\tfrac{3\tau}{8}) + 2^{-1/2}\sin_2((2 + 4n)\tfrac{3\tau}{8}) = \sin(\tfrac{\pi}{4}) + 2^{-1/2}\sin(\tfrac{3\pi}{2}) = 0 \, . 
\end{align*}
Thus it suffices to find an unbounded subsequence $n_k \in \Z_{\geq 0}$ such that 
\begin{align*}
\lim_{k \rightarrow +\infty}  \sin_2((4n_k + 3)\lambda_1) + 2^{-1/2}\sin_2((4n_k + 2)\lambda_1) = 0 \, . 
\end{align*}
Since $4\lambda_1$ is an irrational multiple of $\tau/2$ by Corollary \ref{CorollaryQuadraticIrrationalTrigonometricNumbers}, then we may take an unbounded sequence $n_k \in \Z_{\geq 0}$ such that $4n_k\lambda_1 \rightarrow 0$ in $\R/\tfrac{\tau}{2}\Z$ as $k \rightarrow +\infty$ and get
\begin{align*}
\lim_{k \rightarrow +\infty}  \sin_2((4n_k + 3)\lambda_1) + 2^{-1/2}\sin_2((4n_k + 2)\lambda_1) = \sin_2(3\lambda_1) + 2^{-1/2}\sin_2(2\lambda_1) \, .  
\end{align*}
Finally, note that $\cos_2(\lambda_1) = \tfrac{1}{2\sqrt{2}}$ and $\sin_2(\lambda_1) = \tfrac{\sqrt{7}}{2\sqrt{2}}$, so that 
\begin{align*}
\sin_2(3\lambda_1) + 2^{-1/2}\sin_2(2\lambda_1) &= 3 \cos_2^2(\lambda_1)\sin_2(\lambda_1) - \sin_2^3(\lambda_1) + \sqrt{2} \cos_2(\lambda_1) \sin_2(\lambda_1) \\
&= 3 \tfrac{\sqrt{7}}{16\sqrt{2}} - \tfrac{7\sqrt{7}}{16\sqrt{2}} + \sqrt{2} \tfrac{\sqrt{7}}{8} = 0 \, . 
\end{align*}
\end{proof}

\appendix

\section{Rational linear dependence of trigonometric numbers}
\label{Rational linear dependence of trigonometric numbers}

In proving Proposition \ref{PropositionTrigonometricEquationsRationalSolutions} we make use of two results regarding rational linear dependence regarding trigonometric numbers of the form $\cos(\tfrac{a\pi}{b})$ for coprime integers $a, b \in \Z$. We first establish the algebraic degree of such numbers over $\Q$. To state it we will denote by $\phi$ the Euler totient function.

\begin{lemma}
\label{LemmaQuadraticIrrationalValuesOfCosineOfARationalAngle}
Let $a, b \in \Z$ be coprime integers. Then the algebraic degree of $\cos(\tfrac{a\pi}{b})$ over $\Q$ is $1$ if $b = 1$ and $\phi(2b)/2$ if $b \geq 2$.
\end{lemma}
\begin{proof}
Since $\cos(\tfrac{a\pi}{b}) = (\e^{i\frac{a\pi}{b}} + \e^{-i\frac{a\pi}{b}})/2$ and $z = \e^{i\frac{a\pi}{b}}$ solves the equation
\begin{align*}
z^2 - 2 \cos(\tfrac{a\pi}{b}) z + 1 = 0
\end{align*}
then $[\Q(\e^{i\frac{a\pi}{b}}) : \Q(\cos(\tfrac{a\pi}{b}))] \leq 2$. Note that this degree is $1$ if and only if $\e^{i\frac{a\pi}{b}}$ is real, meaning $b = 1$. If on the other hand the degree is $2$ then 
$$[\Q(\e^{i\frac{a\pi}{b}}) : \Q] = [\Q(\e^{i\frac{\pi}{b}}) : \Q] = \phi(2b)$$
since $a, b$ are coprime. Thus we get that the algebraic degree of $\cos(\tfrac{a\pi}{b})$ is
\begin{align*}
[\Q(\cos(\tfrac{a\pi}{b})) : \Q] = \frac{[\Q(\e^{i\frac{a\pi}{b}}) : \Q]}{[\Q(\e^{i\frac{a\pi}{b}}) : \Q(\cos(\tfrac{a\pi}{b}))]} = \begin{cases}
    1 &\mbox{ if } b = 1 \\
    \phi(2b)/2 &\mbox{ if } b \geq 2 \, . 
\end{cases}
\end{align*}
\end{proof}

An immediate consequence is that we can for example find all rational and quadratic irrational trigonometric numbers $\cos(\tfrac{a\pi}{b})$. The following result can be found in Niven's book \cite[Corollary 3.12]{NivenIrrationalNumbers}, which should also be credited to the prior work of Lehmer and Olmstead.

\begin{corollary}[Niven's Theorem]
\label{CorollaryNivensTheorem}
Let $a, b \in \Z_{\geq 1}$ such that $\gcd(a, b) = 1$. Then the following are equivalent:
\begin{enumerate}
    \item $\cos(\tfrac{a\pi}{b}) \in \Q$,
    \item $b \leq 3$.
\end{enumerate}
In other words, the only rational values of $\cos(\tfrac{a\pi}{b})$ are $\pm 1$ and $\pm \frac{1}{2}$.
\end{corollary}
\begin{corollary}
\label{CorollaryQuadraticIrrationalTrigonometricNumbers}
Let $a, b \in \Z_{\geq 1}$ such that $\gcd(a, b) = 1$. Then $\cos(\tfrac{a\pi}{b})$ is quadratic irrational if and only if $b \in \{4, 5, 6\}$. In other words, the only quadratic irrational values of $\cos(\tfrac{a\pi}{b})$ are $\pm \tfrac{1}{\sqrt{2}}, \pm \tfrac{1}{4} \pm \tfrac{\sqrt{5}}{4}$ and $\pm \tfrac{\sqrt{3}}{2}$.
\end{corollary}
A much more subtle result that we need to prove Proposition \ref{PropositionTrigonometricEquationsRationalSolutions} is the following generalisation of Niven's Theorem due to Berger, see \cite[Theorem 1.2]{BergerOnLinearIndependenceOfTrigonometricNumbers}.
\begin{theorem}[Berger]
\label{TheoremBerger}
Let $a_1, b_1, a_2, b_2 \in \Z_{\geq 1}$ such that $\gcd(a_1, b_1) = \gcd(a_2, b_2) = 1$ and $\tfrac{a_1}{b_1} \pm \tfrac{a_2}{b_2} \notin \Z$. Then the following are equivalent:
\begin{enumerate}
    \item $\cos(\tfrac{a_1\pi}{b_1}), \cos(\tfrac{a_2\pi}{b_2})$ are linearly dependent over $\Q$,
    \item $b_1, b_2 \leq 3$ or $(b_1, b_2) = (5, 5)$.
\end{enumerate}
\end{theorem}
If we have a solution $\lambda = \tfrac{a\tau}{2b} \in (0, \tfrac{\tau}{2})$ with $a,b$ coprime to
\begin{align*}
\sin_q(\tfrac{(r + 1)a \tau}{2b}) + q^{-1/2}\sin_q(\tfrac{r a \tau}{2b}) = 0 \, , 
\end{align*}
or equivalently 
\begin{align}
\label{EqTheMainSineEquation}
\sin(\tfrac{(r + 1)a \pi}{b}) + q^{-1/2}\sin(\tfrac{r a \pi}{b}) = 0 \, , 
\end{align}
then multiplying by $\sin(\tfrac{(r + 1)a \pi}{b}) - q^{-1/2}\sin(\tfrac{r a \pi}{b})$ and using the formula $2\sin^2(t) = 1 - \cos(2t)$ yields the equation
\begin{align*}
\cos(2(r+1)\tfrac{a \pi}{b}) - q^{-1} \cos(2r\tfrac{a \pi}{b}) - (1 - q^{-1}) = 0 \, . 
\end{align*}
By Theorem \ref{TheoremBerger} with $a_1/b_1 = 2(r+1)a\pi/b$ and $a_2/b_2 = 2ra\pi/b$ we can restrict the possible solutions to the following cases:
\begin{enumerate}
    \item $\tfrac{2(2 r + 1)a}{b} = \tfrac{a_1}{b_1} + \tfrac{a_2}{b_2} \in \Z$, which implies that there is a $n \in \Z$ such that $r = (nb - 2)/4$. Since $0 \leq r \leq b-1$, we may assume $n \in \{1, 2, 3\}$.
    \item $\tfrac{2a}{b} = \tfrac{a_1}{b_1} - \tfrac{a_2}{b_2} \in \Z$, meaning $b = 2$.
    \item $5 \tfrac{2(r + 1)a}{b}, 5 \tfrac{2ra}{b} \in \Z$, which implies 
    \begin{align*}
        \frac{10a}{b} = 5 \frac{2(r + 1)a}{b} -  5 \frac{2ra}{b} \in \Z \, ,
    \end{align*}
    meaning $b \in \{2, 5, 10\}$.
    \item there are $1 \leq n_1, n_2 \leq 3$ such that $n_1 \tfrac{2(r + 1)a}{b}, n_2 \tfrac{2ra}{b} \in \Z$, which implies that 
    \begin{align*}
        \frac{2n_1}{b} = n_1 \frac{2(r + 1)a}{b} -  n_1 \frac{2ra}{b} \in \Z
    \end{align*}
    if $n_1 = n_2$ and 
    \begin{align*}
        \frac{2n_1n_2}{b} = n_1 n_2 \frac{2(r + 1)a}{b} -  n_1 n_2 \frac{2ra}{b} \in \Z
    \end{align*}
    if $n_1 \neq n_2$, meaning $b \in \{2, 3, 4, 6, 12\}$.
\end{enumerate}
To prove Proposition \ref{PropositionTrigonometricEquationsRationalSolutions}, we check each of these cases. 

\begin{proof}[Proof of Proposition \ref{PropositionTrigonometricEquationsRationalSolutions}]
For case (1) above we plug in $r = (nb - 2)/4$ into Equation \ref{EqTheMainSineEquation} and split $an$ into congruence classes modulo 4 to get
\begin{align*}
\sin((\tfrac{nb - 2}{4} + 1)\tfrac{a\pi}{b}) &+ q^{-1/2} \sin(\tfrac{nb - 2}{4}\tfrac{a\pi}{b}) = \sin(\tfrac{a\pi}{2b} + \tfrac{na\pi}{4}) - q^{-1/2} \sin(\tfrac{a\pi}{2b} - \tfrac{na\pi}{4}) \\
&= \begin{dcases} 
(-1)^k(1 - q^{-1/2}) \sin(\tfrac{a\pi}{2b}) &\mbox{ if } na = 4k \\
\tfrac{(-1)^k}{\sqrt{2}} \Big( (1 + q^{-1/2})\cos(\tfrac{a\pi}{2b}) + (1 - q^{-1/2})\sin(\tfrac{a\pi}{2b}) \Big) &\mbox{ if } na = 4k + 1 \\
(-1)^k(1 + q^{-1/2}) \cos(\tfrac{a\pi}{2b}) &\mbox{ if } na = 4k + 2 \\
\tfrac{(-1)^k}{\sqrt{2}} \Big( (1 + q^{-1/2})\cos(\tfrac{a\pi}{2b}) - (1 - q^{-1/2})\sin(\tfrac{a\pi}{2b}) \Big)  &\mbox{ if } na = 4k + 3 
\end{dcases} 
\end{align*}
for some $k \in \Z$. This expression does not vanish when $an \equiv 0,2 \mod 4$ since $\sin(\tfrac{a\pi}{2b}), \cos(\tfrac{a\pi}{2b}) > 0$. When $an \equiv 1, 3 \mod 4$ we get that Equation \ref{EqTheMainSineEquation} can be written as 
\begin{align*}
\tan(\tfrac{a\pi}{2b}) = \pm \frac{\sqrt{q} + 1}{\sqrt{q} - 1} \, , 
\end{align*}
which implies that 
\begin{align*}
\cos(\tfrac{a\pi}{b}) = \frac{1 - \tan^2(\tfrac{a\pi}{2b})}{1 + \tan^2(\tfrac{a\pi}{2b})} = - \frac{\sqrt{q}}{q + 1} \, .
\end{align*}
Since $q \geq 2$, we see that $\pi^{-1}\arccos(- \frac{\sqrt{q}}{q + 1})$ is irrational by Corollary \ref{CorollaryQuadraticIrrationalTrigonometricNumbers}. We conclude that there are no solutions to Equation \ref{EqTheMainSineEquation} in case (1).

Cases (2), (3) and (4) one checks by hand, plugging in every possible value of $\tfrac{a\pi}{b}$ into Equation \ref{EqTheMainSineEquation} for $1 \leq a \leq b$ coprime and $0 \leq r \leq b - 1$ when $b \in \{2, 3, 4, 5, 6, 10, 12\}$. The solutions are 
\begin{align*}
\sin(\tfrac{(2 + 1)3 \pi}{4}) + 2^{-1/2}\sin(\tfrac{2 \cdot 3 \pi}{4}) &= 0 \, , \\
\sin(\tfrac{(9 + 1)5 \pi}{12}) + 2^{-1/2}\sin(\tfrac{9 \cdot 5 \pi}{12}) &= 0 \, , \\
\sin(\tfrac{(9 + 1)11 \pi}{12}) + 2^{-1/2}\sin(\tfrac{9 \cdot 11 \pi}{12}) &= 0 \, , \\
\sin(\tfrac{(4 + 1)5 \pi}{6}) + 3^{-1/2}\sin(\tfrac{4 \cdot 5 \pi}{6}) &= 0 \, .
\end{align*}
\end{proof}


\printbibliography

{\small\textsc{Department of Mathematics, Chalmers and University of Gothenburg, Gothenburg, Sweden \\
 \textit{Email address}: {\tt bylehn@chalmers.se}}}

\end{document}